%% file: Longform_calculus_of_sub-differentials_output.tex
\documentclass{article}
\input{header}

\input{preamble.sty}
\title{Bilateral facial reduction: qualification-free subdifferential calculus and exact duality}
\author{Matthew S. Scott}
\begin{document}
\maketitle
\begin{abstract}
Qualification conditions (also termed constraint qualifications) help avoid pathological behavior at domain boundaries in convex analysis. By generalizing facial reduction from conic programming to general
convex programs of the
form $f(x) + g(Ax)$, we provide qualification-free
generalizations of several key results: an exact Fenchel-Rockafellar
dual, 
KKT optimality conditions, 
an attained infimal convolution for the conjugate of a sum,
subdifferential sum and chain rules, and normal cones of intersections.
All our results reduce seamlessly to their original formulations when qualification conditions hold.

The core insight is that for a sum of
two convex functions, there is an affine subspace---the joint supporting subspace---
that contains the feasible region, and such that qualification conditions
hold when restricting the effective domain of each function to it. We offer a number of characterizations for the joint supporting subspace, including 
one that obtains the affine subspace via iterative, bilateral reduction between the two domains.
In our proofs, which are self-contained, we develop
a structured induction on faces 
where inductive steps are associated with normal vectors nested in 
supporting subspaces (a generalization of supporting hyperplanes). 
With this tool, we characterize the facial structure of the difference of two convex
sets from the 
facial structures of the individual convex sets.
\end{abstract}
\section{Introduction}
\label{loc:body.introduction}
In convex analysis, many results require so-called \emph{qualification conditions}
(sometimes called constraint qualifications),
such as 
\begin{equation}
\label{eq:intro:ri}
\rint\dom(f) \cap \rint\dom(g) \neq \emptyset,
\end{equation}
This condition can be understood as ``avoiding the boundary"; it ensures that the feasible region $\dom(f + g)$ 
is not contained in the relative boundary of either $\dom(f)$ or $\dom(g)$. 
When \Cref{eq:intro:ri} fails to hold, the geometry of the boundary
of the convex sets $\dom(f)$ and $\dom(g)$ starts to matter, and such cases may require
special treatment.

Qualification conditions are not merely theoretical artifacts; they are
justified by counterexamples, and therefore
reflect more or less accurately the true scope of applicability of the
associated formulas.
We introduce one such counterexample for the sum rule of subdifferentials.
(We will later show how our theory resolves this case.)
Take the function $f(x) := -\sqrt{x}$ when $x \ge 0$ and $f(x) := \infty$ otherwise.
Consider the indicator function $g(x)= \indicator_{(-\infty,  0]}(x)$, 
which is $0$ on $(-\infty, 0]$ and $\infty$ elsewhere.
Naive application of the sum rule yields 
\begin{equation*}
\partial (f+g)(0) =  \partial f(0) + \partial g(0)= \emptyset + \mathbb{R}_+ =
\emptyset,
\end{equation*}
when in fact, $\partial (f+g)(0)= \partial \indicator_{\{0\}}(0) = \mathbb{R}$, so the sum
rule does not hold in this case.
There is no contradiction in this failure because 
the qualification condition in \Cref{eq:intro:ri} is not satisfied in this case (otherwise, the sum rule would be guaranteed by \cite[Theorem 23.8]{rockafellarConvexAnalysis1970}).
Indeed, the relative interiors
of the domains are $(-\infty,  0)$ and $(0, \infty)$,
and do not intersect.

Our approach generalizes facial reduction from the conic setting to general
convex programs
of the form $f(x) + g(Ax)$---let us here consider the special case $f(x) + g(x)$. We identify a single affine subspace $T_a \subseteq
\mathbb{R}^n$ that simultaneously regularizes both functions and contains the feasible region.
The joint supporting subspace $T_a$ is defined for any two convex sets (see \Cref{loc:the_joint_supporting_subspace.statement}), which are usually the domains of two convex functions. 
We can then reduce the effective domains of each function to $T_a$: $f' := f + \indicator_{T_a}, g' := g + \indicator_{T_a}$. For these reduced functions, qualification conditions are then guaranteed to hold, while preserving the original problem in the sense that $f+g \equiv f' + g'$. 
With these reduced functions we modify the sum rule to resolve our previously encountered
counterexample, by using the reduced functions on the r.h.s: $\partial (f+g)(x) = \partial f'(x) + \partial g'(x)$. In our counter-example, the joint supporting subspace is $T_a = \{0\}$, and the modified
sum rule yields 
\begin{equation*}
\partial f'(0) + \partial g'(0) =
\partial \indicator_{\{0\}}(0) + \partial \indicator_{\{0\}}(0) = \mathbb{R},
\end{equation*}
which indeed matches $\partial(f+g)(0)$.
We note that the reduced functions $f'$, $g'$ are trivial in this example
only because the feasible region is itself trivial. In general, our reduction will preserve the functions' behaviour inside---and to some extent outside---the feasible region.

In this manner, in \Cref{loc:body.main_results} we generalize
a collection of results from convex analysis,
with the classical formulations trivially recovered when qualifications hold (see~\Cref{loc:recovery_of_the_classical_results.statement}).
Specifically, we find a modified Fenchel-Rockafellar dual for which strong duality always holds,
problem-independent Karush-Kuhn-Tucker (KKT)
optimality conditions, 
always-attained infimal convolution for the convex conjugates of sums of
pairs of functions,
qualification-free subdifferential sum and chain rules, 
and a characterization of the normal cones
of the intersection of two convex sets.

We provide a number of characterizations for $T_a$ in \Cref{loc:body.the_joint_supporting_subspace}. We specify the joint supporting subspace $T_a$ in terms of the difference of domains $\dom(f)-\dom(g)$, where the set difference is $C-D := \{c-d \mid c \in C,  d \in D\}$. Other characterizations relate $T_a$ to minimal faces of the domains containing the feasible region (\Cref{loc:characterization_through_generated_faces.statement}, \Cref{loc:joint_supporting_subspace_as_affine_hull.statement}, and \Cref{loc:joint_supporting_subspace_from_difference_of_faces.statement}). One of our characterizations (\Cref{loc:iterative_bilateral_facial_reduction.statement}) consists of an iterative reduction procedure which recovers $T_a$ and terminates in a number of steps no more than the ambient dimension. This reduction can be performed from any point in the feasible region, using so-called nested normal cones (see \Cref{loc:nested_normals.statement}). We show that the reduction at a point does not in fact depend on the choice of the point, which means that the reduction need only be performed once for a given optimization problem. It is known that facial reduction in the conic setting leads to a reduction of the size of the problem and improved numerical
stability~\cite{drusvyatskiyManyFacesDegeneracy2017}, which gives reason to believe that our generalization may have similar benefits.

A contribution of the present work
is in our proof technique, found in \Cref{loc:body.a_framework_for_local_face_lattices}. The
central objects are \emph{nested normals} (\Cref{loc:nested_normals.statement}) and \emph{supporting subspaces}
(\Cref{loc:supporting_subspace_at_s.statement}), which together describe the local lattice
of faces of a convex set. We establish
a structured induction via nested normals,
which traverses all supporting subspaces, 
with induction steps associated with nested normals (see~\Cref{loc:induction_by_nested_normal.statement}).
With nested normals and our structural induction, we characterize the minimal supporting subspace containing 
a set (see \Cref{loc:the_generated_supporting_subspace_is_the_only_supporting_subspace_with_no_nested_normals.statement}). We also characterize the facial structure of a difference of
convex sets $C-D$ at $0$ exclusively from the local facial structure of the individual convex sets $C$ and $D$ at $C \cap D$ (see \Cref{loc:characterization_of_the_faces_of_differences_of_convex_sets.statement}). These results enable us to prove our characterizations of the
joint supporting subspace that are presented in \Cref{loc:body.the_joint_supporting_subspace}.
\subsection{Paper outline}
\label{loc:introduction_for_sub:differentials_additivity.statement.paper_outline}
\Cref{loc:body.prior_works} surveys prior works, notably facial reduction in conic programming, limiting approaches to exact subdifferential calculus, and the lexicographic characterization of the faces of convex sets. In \Cref{loc:body.notation_and_preliminaries} we introduce some notation. In \Cref{loc:body.the_joint_supporting_subspace} we state analytic and constructive characterizations of the joint supporting subspace, with proofs deferred to \Cref{loc:body.proofs}. In \Cref{loc:qualification_conditions_always_hold_in_the_joint_supporting_subspace.statement} we state the main property of the joint supporting subspace: that qualification conditions hold when localizing to it.
In \Cref{loc:body.main_results}, we state and prove generalizations of central results in convex analysis. In \Cref{loc:body.a_framework_for_local_face_lattices}, we build the core theoretical framework. In \Cref{loc:body.proofs}, we prove the results of \Cref{loc:body.the_joint_supporting_subspace} with the tools of \Cref{loc:body.a_framework_for_local_face_lattices}.
\section{Prior works}
\label{loc:body.prior_works}
Constraint qualifications (CQs) in optimization started with Slater's condition~\cite{slaterLagrangeMultipliersRevisited1959} and much work has since been aimed at finding relaxations of these conditions~\cite{mangasarianFritzJohnNecessary1967, abadieNonlinearProgramming1967, robinsonFirstOrderConditions1976}. 
Notably, a qualification condition using relative interiors (the weak Slater) was introduced by~\textcite{rockafellarConvexAnalysis1970}, who used the same qualification conditions for optimization
problems and subdifferential calculus, which
enabled our unified treatment.

There are a number of qualification-free alternatives to classical results, which come at the cost of some added complexity. One approach is to consider approximations of the core objects, with qualification-free exact results obtained in the limit of increasingly accurate approximation. This type of result includes $\varepsilon-$subdifferentials~\cite{hiriarturruty_SubdifferentialCalculusUsing_1993}, and fuzzy sum rules~\cite{ioffeFuzzyPrinciplesCharacterization1998, jules_FormulasSubdifferentialsSums_2002}.
Similarly, sequential optimality conditions have been formulated~\cite{thibaultSequentialConvexSubdifferential1997, jeyakumarInequalitySystemsGlobal1996}. Sequential approaches are related to the broader theory of variational analysis.

Another kind of qualification-free result is obtained by considering the geometry of
domain boundaries, with the notable example of facial reduction in conic optimization. Our work is the first, to our knowledge, to extend this geometric methodology to qualification-free subdifferential calculus.

In this vein, a qualification-free necessary and sufficient condition was formulated
by \textcite{ben-talCharacterizationOptimalityConvex1976}, and extended to
the nonsmooth case by \textcite{wolkowiczGeometryOptimalityConditions1980}, which achieve qualification-free characterization of minimizers by adding a term to the KKT conditions. Our qualification-free KKT condition in \Cref{loc:optimality_condition_for_the_exact_fr_dual.statement}
is formulated for programs of the form $f(x) + g(Ax)$, instead of inequality-constrained convex programs,
and uses a reduction instead of
an additional term.

Our theory extends Facial Reduction (FR) from conic programs to programs of
the form $f(x) + g(Ax)$. The field of facial reduction started by generalizing preprocessing
ideas from linear programming
to the setting of conic programming~\cite{massamOptimalityConditionsConeConvex1979, borweinFacialReductionConeconvex1981, borweinRegularizingAbstractConvex1981},
with subsequent developments 
providing efficient computational methods.
One computational method consists in iteratively reducing the dimension of the
problem~\cite{borweinRegularizingAbstractConvex1981, luoDualityResultsConic1997, wakiFacialReductionAlgorithms2013}
until the minimal face is reached, an approach
that we generalize in \Cref{loc:iterative_bilateral_facial_reduction.statement}.
Ideas of facial reduction can alternatively be framed as
\emph{exact dual problems}~\cite{borweinCharacterizationOptimalityAbstract1981},
with important applications to semidefinite programming~\cite{ramanaExactDualityTheory1997, patakiStrongDualityConic2013}.

From convex geometry, a line of work that enabled the development of our theory was the \emph{lexicographical
characterization of faces}~\cite{martinez-legazLexicographicalSeparationRn1987, martinez-legaz_LEXICOGRAPHICALCHARACTERIZATIONFACES_}. 
It turned out that the
lexicographical framework was not flexible enough for our needs, prompting us to develop an analogous self-contained theory, recovering the lexicographical characterization of convex faces
in~\Cref{loc:completeness_of_the_composition_of_nested_normals.statement} in different terms.
Specifically, our framework based on nested normals allows us to specify
a special kind of structural induction on supporting subspaces (and thereby faces)
in \Cref{loc:induction_by_nested_normal.statement}.
This induction is powerful because each inductive step is associated
with a single nested normal. Additionally, the use of nested normals allows
us to specify a reduction procedure that uses nested normal cones at each
step, instead of individual nested normals (see~\Cref{loc:simplified_iterative_bilateral_facial_reduction.statement} and \Cref{loc:iterative_bilateral_facial_reduction.statement}). This formulation provides structural insight and allows, in most cases, termination in fewer steps by dropping many dimensions at once in some steps.

Shortly after the first version of our work was made public, an independent paper by \textcite{linFacialReductionNice2025} was made available which specifies a reduction for convex programs of the form $f(x) + g(x)$. Their work has an algorithmic focus: they show feasibility detection and fast reductions with extended duals, by defining a notion of niceness for the domains of the convex functions. In contrast, our work is more geometric, providing analytical characterizations of the joint supporting subspace in \Cref{loc:body.the_joint_supporting_subspace} and qualification-free results stated with this object in \Cref{loc:body.main_results}, including a reduction for programs in the slightly more general form $f(x) + g(Ax)$, and subdifferential calculus sum and chain rules. Additionally, the proof techniques of the two papers seem remarkably distinct: they lift to cones, whereas we introduce nested normals in the original space.
\section{Notation and Preliminaries}
\label{loc:body.notation_and_preliminaries}
For two sets $A, B \subseteq \mathbb{R}^n$, $A+B := \{a+b \mid a \in A, b \in B\}$ is the Minkowski sum, with $A-B := \{a-b \mid a \in A, b \in B\}$,
and for $z \in \mathbb{R}^n$, let $A+z := A + \{z\}$.
For a set $S \subseteq \mathbb{R}^n$, we denote
by $\indicator_S(x)$ the indicator function which is
$0$ if $x \in S$ and $\infty$ otherwise.

For a set $S \subseteq \mathbb{R}^n$, denote by $\conv(S)$ its convex hull (smallest containing convex set). For $y_1, y_2 \in \mathbb{R}^n$, let $[y_1, y_2] := \conv \{y_1, y_2\}$, and $]y_1, y_2[ := \conv \{y_1, y_2\} \setminus \{y_1, y_2\}$.
We also denote $\cone(S)$ for the conical hull of a set; the smallest cone
containing $S$.

Let $C \subseteq \mathbb{R}^n$ be convex. We denote by $\rint(C)$ the relative interior of $C$. A subset $F \subseteq C$ is defined to be a \emph{face} of $C$ when $\forall y, z\in C$, $(\exists x \in ]y, z[ \cap F \implies y, z \in F)$.
That is, a face $F \subseteq C$ is a subset of $C$ for which there is no line segment that ``goes through" $F$ while remaining in $C$. 
A face $F \subseteq C$ is \emph{exposed} when there is a $u \in \mathbb{R}^n \setminus \{0\}$ such that $F = \argmax_{c \in C} \langle c, u\rangle$.

For any set $S \subseteq C$, we denote by $F_C(S)$ the face of $C$ \emph{generated} by $S$;
it is the smallest face of $C$ that contains $S$. (It is also the intersection
of all faces containing $S$, which, it can be shown, is itself a face.)

For $x \in \mathbb{R}^n$ and a convex set $C \subseteq \mathbb{R}^n$,
the normal cone of $C$ at $x$ is $N_C(x) = \{y \mid \forall c \in C, \langle y, c - x \rangle \le 0\}$
if $x \in C$, and $N_C(x)= \emptyset$ otherwise.
We denote $\mathbb{R}_{+} := [0, \infty)$, $\mathbb{R}_{++}:= (0, \infty)$, and similarly for $\mathbb{R}_{-}$ and $\mathbb{R}_{--}$.
The extended reals are $\bar{\mathbb{R}} = \mathbb{R} \cup \{-\infty, \infty\}$. Given a convex function
$f:\mathbb{R}^n \to \bar{\mathbb{R}}$
we denote by $\dom(f)$ the effective domain of $f$; $\dom(f) := \{x \in
\mathbb{R}^n \mid f(x) \neq \infty\}$.
The function $f$ is \emph{proper} if it is nowhere $-\infty$ and has nonempty effective domain. 
The subdifferential is
\begin{equation*}
\partial f(x) := \{v \in \mathbb{R}^n \mid \forall y \in \mathbb{R}^n, f(x) + \langle v, y-x\rangle \le f(y)\},
\end{equation*}
with $\partial f(x) = \emptyset$ when $x \notin \dom(f)$.
The convex conjugate is $f^*(y):= \sup_{x \in \mathbb{R}^n}\langle y, x\rangle - f(x)$.
\section{The joint supporting subspace}
\label{loc:body.the_joint_supporting_subspace}
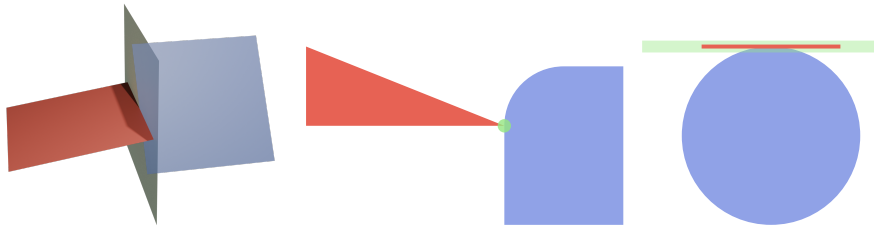
\begin{figure}[htbp]
    \centering
    \begin{subfigure}{0.32\textwidth}
        \centering
        \includegraphics[width=\textwidth]{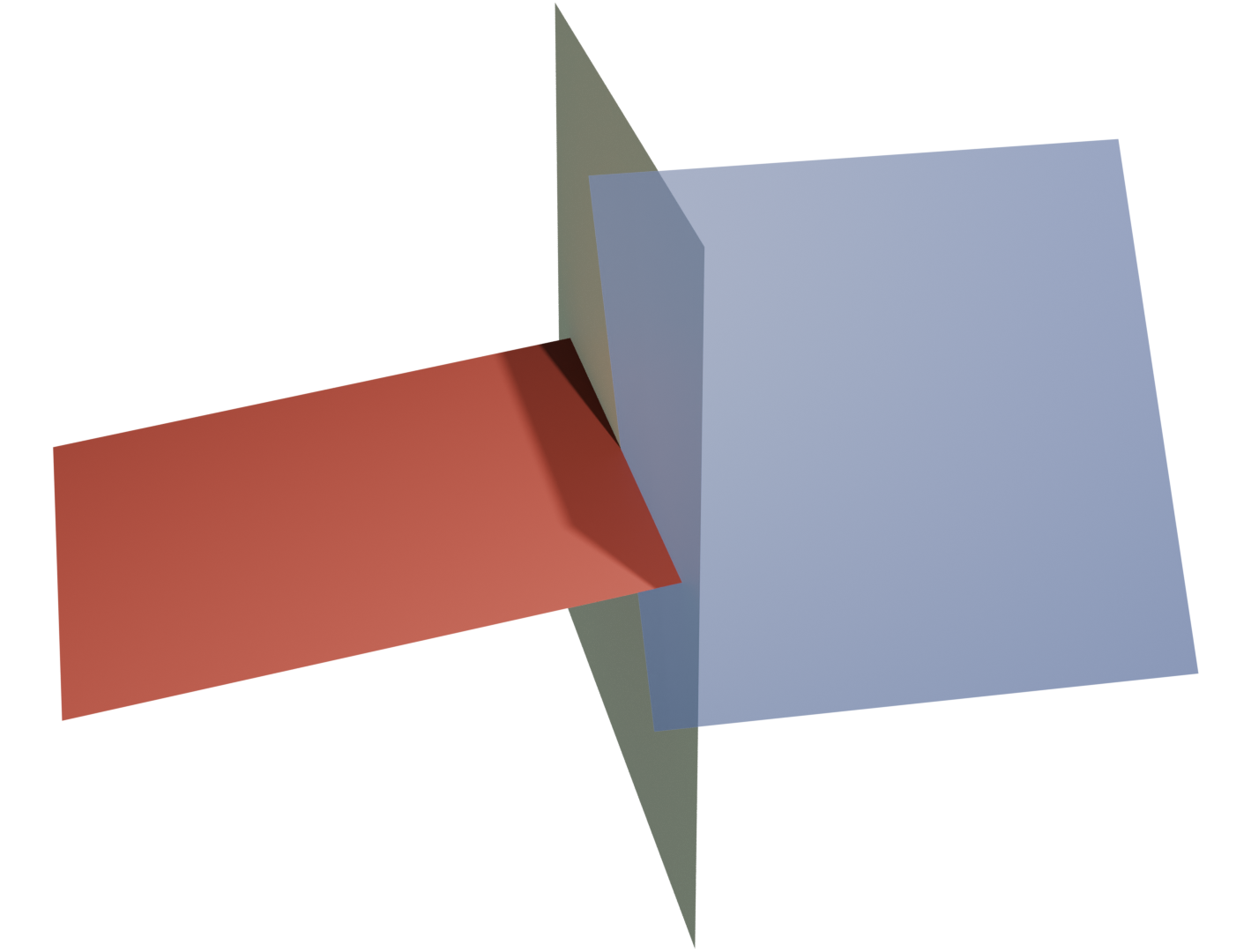}
    \end{subfigure}
    \hfill
    \begin{subfigure}{0.32\textwidth}
        \centering
        \tikzset{every picture/.style={line width=0.75pt}} 

        \begin{tikzpicture}[x=0.75pt,y=0.75pt,yscale=-1,xscale=1]

        \draw  [draw opacity=0][fill={rgb, 255:red, 144; green, 162; blue, 231 }  ,fill opacity=1 ] (250,100) .. controls (250,83.43) and (263.43,70) .. (280,70) -- (310,70) -- (310,150) -- (250,150) -- cycle ;
        \draw  [draw opacity=0][fill={rgb, 255:red, 231; green, 98; blue, 84 }  ,fill opacity=1 ][line width=0.75]  (150,60) -- (250,100) -- (150,100) -- cycle ;
        \draw  [draw opacity=0][fill={rgb, 255:red, 149; green, 231; blue, 129 }  ,fill opacity=0.8 ] (246.63,100) .. controls (246.63,98.14) and (248.14,96.63) .. (250,96.63) .. controls (251.86,96.63) and (253.38,98.14) .. (253.38,100) .. controls (253.38,101.86) and (251.86,103.38) .. (250,103.38) .. controls (248.14,103.38) and (246.63,101.86) .. (246.63,100) -- cycle ;
        \end{tikzpicture}
    \end{subfigure}
    \hfill
    \begin{subfigure}{0.32\textwidth}
        \centering
        \tikzset{every picture/.style={line width=0.75pt}} 

        \begin{tikzpicture}[x=0.75pt,y=0.75pt,yscale=-1,xscale=1]

            \draw  [draw opacity=0][fill={rgb, 255:red, 144; green, 162; blue, 231 }  ,fill opacity=1 ] (150,105) .. controls (150,80.15) and (170.15,60) .. (195,60) .. controls (219.85,60) and (240,80.15) .. (240,105) .. controls (240,129.85) and (219.85,150) .. (195,150) .. controls (170.15,150) and (150,129.85) .. (150,105) -- cycle ;
            \draw [color={rgb, 255:red, 149; green, 231; blue, 129 }  ,draw opacity=0.4 ][line width=4.5]    (130,60) -- (250,60) ;
            \draw [color={rgb, 255:red, 231; green, 98; blue, 84 }  ,draw opacity=1 ][line width=1.5]    (160,60) -- (230,60) ;
        \end{tikzpicture}
        \end{subfigure}
    \caption{In each figure, the green subspace is the joint supporting subspace for the red and blue convex sets.}
    \label{fig:combined}
\end{figure}
In this section we investigate the
following object.
\begin{definition}[The joint supporting subspace]
\label{loc:the_joint_supporting_subspace.statement}
The \emph{joint supporting subspace} of two convex sets $C, D \subseteq \mathbb{R}^n$ is
\begin{equation*}
T(C, D) := \Span F_{C-D}(0).
\end{equation*}
We also define its affine counterpart: $T_a(C, D) := T(C,D) + C \cap D$.
\end{definition}
We note that when $C \cap D =  \emptyset$, we have $T_a = \emptyset$ also.
This is an intentional convention which takes care of the infeasible case
in \Cref{loc:body.main_results}.
The Minkowski sum $(+ C \cap D)$ in the definition of $T_a$ is in fact
an affine translation. This follows from the next corollary.
\begin{corollary}[Containment of the intersection]
\label{loc:containment_of_the_intersection_in_the_joint_supporting_subspace.statement}
Let $C, D \subseteq \mathbb{R}^n$ be convex. Then
\begin{equation*}
C \cap D - C \cap D \subseteq T(C, D).
\end{equation*}
\end{corollary}
All proofs for results presented in this section are deferred to \Cref{loc:body.proofs}, because they require
the machinery introduced in \Cref{loc:body.a_framework_for_local_face_lattices}. \Cref{loc:containment_of_the_intersection_in_the_joint_supporting_subspace.statement} follows from \Cref{loc:decomposition_of_supporting_subspaces_of_differences.statement}; for details, \hyperlink{loc:containment_of_the_intersection_in_the_joint_supporting_subspace.proof}{the proof} can be found at the end of \Cref{loc:narrative_for_conditional_subdifferentials.statement.nested_normals_of_a_difference_of_sets}.

From \Cref{loc:containment_of_the_intersection_in_the_joint_supporting_subspace.statement}, we see that picking any $x \in C \cap D$, we have $T_a = T + x$, independently of the choice of $x$. \Cref{loc:containment_of_the_intersection_in_the_joint_supporting_subspace.statement} also implies that
$C \cap D \subseteq T_a$, showing that $T_a$ is the unique translation
of $T$ that contains $C \cap D$.

In this section we provide two main characterizations of the joint supporting subspace.
The next result builds towards the first characterization, describing the
relation of the joint supporting subspace with the individual convex sets.
\begin{lemma}[The joint supporting subspace reveals faces]
\label{loc:the_joint_supporting_subspace_reveals_faces.statement}
For convex sets $C, D \subseteq \mathbb{R}^n$, and $T_a=T_a(C, D)$,
\begin{equation*}
T_a \cap C = F_C (C  \cap D).
\end{equation*}
\end{lemma}
Next we introduce the \emph{generated
supporting subspace}, which is central to our theory. This is made clear
in \Cref{loc:body.a_framework_for_local_face_lattices}, where we re-introduce
this object in motivated fashion in~\Cref{loc:generated_supporting_subspace.statement}.
\begin{definition}[Preliminary definition of the generated supporting subspace]
\label{loc:generated_supporting_subspace.basic_definition}
Given a convex set $C \subseteq \mathbb{R}^n$ with $S \subseteq C$ non-empty, define the \emph{supporting subspace} of $C$ \emph{generated} by $S \subseteq C$ to be
\begin{equation*}
H_C(S) := \Span(F_C(S) - F_C(S)).
\end{equation*}
\end{definition}
The next theorem is our first characterization of the joint supporting subspace.
\begin{theorem}[Characterization through generated faces]
\label{loc:characterization_through_generated_faces.statement}
The joint supporting subspace of $C, D \subseteq \mathbb{R}^n$ is
\begin{equation*}
T(C, D) = H_C(C \cap D) + H_D(C \cap D).
\end{equation*}
\end{theorem}
\begin{corollary}[Joint supporting subspace as hull of faces]
\label{loc:joint_supporting_subspace_as_affine_hull.statement}
For convex sets $C, D \subseteq \mathbb{R}^n$,
\begin{equation*}
T_a(C, D) = \aff(F_C(C \cap D) \cup F_D(C \cap D)).
\end{equation*}
\end{corollary}
\begin{corollary}[Joint supporting subspace as difference of faces]
\label{loc:joint_supporting_subspace_from_difference_of_faces.statement}
For convex sets $C, D \subseteq \mathbb{R}^n$,
\begin{equation*}
T(C, D) = \Span(F_C(C \cap D)-F_D(C \cap D)).
\end{equation*}
\end{corollary}
So far, we have characterized the joint supporting subspace with various
analytical formulas. Our next characterization instead specifies how
to obtain the joint supporting subspace constructively, from an iterative
process with a number of steps bounded
by the ambient dimension.
This result, \Cref{loc:iterative_bilateral_facial_reduction.statement}, is best stated in terms of \emph{nested normals},
a concept that we introduce in \Cref{loc:nested_normals.statement},
and motivate throughout \Cref{loc:body.a_framework_for_local_face_lattices}.
For this reason, we defer it to \Cref{loc:body.proofs.second_characterization}. For the moment, we state a direct corollary which does not require this object.
\begin{corollary}[Iterative bilateral facial reduction]
\label{loc:simplified_iterative_bilateral_facial_reduction.statement}
Let $C, D \subseteq \mathbb{R}^n$ be convex sets and fix any
$x \in C \cap D$. Beginning with $T_0  = \mathbb{R}^n$, let
\begin{equation*}
T_{i+1} = T_i \cap (N_{C \cap (T_i + x)}(x) \cap [-N_{D  \cap (T_i + x)}(x)] )^\perp.
\end{equation*}
Let $\ell$ be the smallest number such that 
$T_{\ell+1} = T_\ell$.  Then:
\begin{enumerate}
\item $T_\ell =  T(C, D)$.
\item $\ell \leq n$.
\end{enumerate}
\end{corollary}
We defer \hyperlink{loc:simplified_iterative_bilateral_facial_reduction.proof}{the proof} to \Cref{loc:body.proofs.second_characterization}. 
Now that the joint supporting subspace has been thoroughly identified,
we move on to describing its core property.
Recall that for two convex functions $f, g:\mathbb{R}^n \to \bar{\mathbb{R}}$,
a common qualification condition is that for convex sets $C:= \dom(f), D:=\dom(g)$,
\begin{equation}
\label{eq:ri}
\rint(C) \cap \rint(D) \neq \emptyset.
\end{equation}
It was shown in \cite[Theorem 11.3]{rockafellarConvexAnalysis1970}
that this notion is equivalent to a \emph{lack of proper separation}
of the domains. We define the notions of \emph{separation}, and \emph{proper separation} below.
We define these notions in affine subspaces, an intuitive generalization
of separation in Euclidean spaces.
\begin{definition}[Separation of sets]
\label{loc:separation_of_sets.statement}
Given an affine subspace $V \subseteq \mathbb{R}^n$, two convex sets
$C, D \subseteq V$ are said to be \emph{separated} in $V$
when there exists some $u \in (V-V) \setminus \{0\}$ such that 
\begin{equation}
\label{eq:sep}
\forall c \in C,  d \in D, \langle u, c\rangle \le \langle u, d\rangle.
\end{equation}
\end{definition}
The concept of \emph{proper separation} is slightly stronger.
\begin{definition}[Proper separation of sets]
\label{loc:proper_separation_of_sets.statement}
For an affine subspace $V \subseteq \mathbb{R}^n$, convex sets $C,D \subseteq \mathbb{R}^n$
are \emph{properly separated} in $V$ when
there exists some $u \in (V-V) \setminus \{0\}$ such that \Cref{eq:sep} holds, and 
additionally such that there exists some pair $(c, d) \in C \times D$ for 
which the inequality in \Cref{eq:sep} is
strict.
\end{definition}
Proper separation excludes the case where the two sets are
merely separated by virtue of being both contained in a single hyperplane.
Since separation follows from proper separation, and lack of proper separation
is equivalent to intersection of relative interiors (by \cite[Theorem 11.3]{rockafellarConvexAnalysis1970}), the contrapositive
statement is that
\begin{equation}
\label{eq:relation_sep}
\text{lack of separation} \implies \text{intersection of relative interiors}.
\end{equation}

In the joint supporting subspace there is always a ``lack of separation"
of the convex sets.
\begin{theorem}[Qualification conditions always hold in the joint supporting subspace]
\label{loc:qualification_conditions_always_hold_in_the_joint_supporting_subspace.statement}
For two convex sets $C, D \subseteq \mathbb{R}^n$ with $T_a := T_a(C, D)$,
$C \cap T_a$ and $D \cap T_a$ are not separated
in $T_a$.
\end{theorem}
\hyperlink{loc:qualification_conditions_always_hold_in_the_joint_supporting_subspace.proof}{The proof} of \Cref{loc:qualification_conditions_always_hold_in_the_joint_supporting_subspace.statement} is the subject of \Cref{loc:body.proofs.regularity_in_the_joint_supporting_subspace}.
\begin{corollary}[rint qualification holds in joint supporting subspace]
\label{loc:rint_qualification_holds_in_joint_supporting_subspace.statement}
For any convex sets $C, D \subseteq \mathbb{R}^n$ such that $C \cap D \neq \emptyset$,
\begin{equation*}
\rint (C \cap T_a)  \cap \rint (D \cap T_a) \neq \emptyset.
\end{equation*}
\end{corollary}
\begin{proof}[\hypertarget{loc:rint_qualification_holds_in_joint_supporting_subspace.proof}Proof of \Cref{loc:rint_qualification_holds_in_joint_supporting_subspace.statement}]

The result follows from \Cref{loc:qualification_conditions_always_hold_in_the_joint_supporting_subspace.statement}
by \Cref{eq:relation_sep}, with the caveat that we defined separation in affine subspaces,
which are not necessarily vector spaces,
while in \cite[theorem 11.3]{rockafellarConvexAnalysis1970} this notion
is defined in $\mathbb{R}^n$. Yet the argument still holds because
the notions of separation, proper separation, and relative interiors
are invariant under translation of $T_a$ to $T$.
For any point $x_0 \in T_a$, $C \cap T_a$ is (properly) separated from 
$C \cap T_a$ iff $C \cap T_a - x_0$ is (properly) separated
from $D \cap T_a - x_0$ in $T$. A similar equivalence holds for intersection
of relative interiors.
\end{proof}
\section{Main results}
\label{loc:body.main_results}
With \Cref{loc:rint_qualification_holds_in_joint_supporting_subspace.statement}
we provide qualification-free generalizations 
by localizing to $T_a$.
The next result shows that when standard qualification conditions hold,
the classical formulations are trivially recovered.
\begin{proposition}[Recovery of the classical results]
\label{loc:recovery_of_the_classical_results.statement}
If the relative interiors
of two convex sets $C, D \subseteq \mathbb{R}^n$
intersect, then
\begin{equation*}
C \cup D \subseteq T_a(C, D).
\end{equation*}
\end{proposition}
\subsection{Subdifferential calculus}
\label{loc:body.main_results.subdifferential_calculus}
\begin{corollary}[Subdifferential sum rule]
\label{loc:qualification:free_subdifferential_additivity.statement}
For $f,g:\mathbb{R}^n\to\bar{\mathbb{R}}$ proper convex, $x  \in \mathbb{R}^n$,
and $T_a = T_a(\dom(f), \dom(g))$,
\begin{equation*}
\partial(f+g)(x) = \partial (f + \indicator_{T_a}) (x) + \partial (g + \indicator_{T_a}) (x).
\end{equation*}
\end{corollary}
We note that in the infeasible case, when $f + g \equiv \infty$, $T_a$ is $\emptyset$, and the result holds trivially. A similar pattern occurs for many results in this section. 
\begin{proof}[\hypertarget{loc:qualification:free_subdifferential_additivity.proof}Proof of \Cref{loc:qualification:free_subdifferential_additivity.statement}]

By
\Cref{loc:rint_qualification_holds_in_joint_supporting_subspace.statement}
applied to $\dom(f)$ and $\dom(g)$,
the result follows by application of
\cite[Theorem 23.8]{rockafellarConvexAnalysis1970} to the functions $f + \indicator_{T_a}$ and $g + \indicator_{T_a}$.
\end{proof}
\begin{corollary}[Normal cone of the intersection of convex sets]
\label{loc:normal_cone_of_the_intersection_of_convex_sets.statement}
For $C, D \subseteq \mathbb{R}^n$ convex, and $x \in C \cap D$, for $T_a=T_a(C,D)$,
\begin{equation*}
N_{C \cap D}(x) = N_{C \cap T_a}(x) + N_{D \cap T_a}(x).
\end{equation*}
\end{corollary}
\begin{proof}[\hypertarget{loc:normal_cone_of_the_intersection_of_convex_sets.proof}Proof of \Cref{loc:normal_cone_of_the_intersection_of_convex_sets.statement}]

With \Cref{loc:rint_qualification_holds_in_joint_supporting_subspace.statement} we can use
\cite[Corollary 23.8.1]{rockafellarConvexAnalysis1970},
from which we find that
\begin{equation*}
N_{C \cap D}(x) = N_{(C \cap T_a) \cap (D \cap T_a)}(x)= N_{C \cap T_a}(x) + N_{D \cap T_a}(x).
\end{equation*}
\end{proof}
In results that follow, we abuse notation by associating
a matrix $A \in \mathbb{R}^{m \times n}$ with the corresponding linear
map $A:\mathbb{R}^n \to \mathbb{R}^m$. For a set $S \subseteq \mathbb{R}^n$,
we denote by $AS$ the set image $A(S)$, and 
by $A^{-1}(\cdot)$ the preimage of the linear map. 
We make no assumption on the
invertibility of the matrix $A$.
\begin{corollary}[Subdifferential chain rule]
\label{loc:generalized_subdifferential_chain_rule.statement}
Let $g:\mathbb{R}^m  \to \bar{\mathbb{R}}$ be a proper convex function,
and $A \in \mathbb{R}^{m  \times n}$ a matrix.
Then for $T := T(\dom(g), \range(A))$,
\begin{equation*}
\partial(g \circ A)(x) = A^T\partial(g + \indicator_{T})(Ax).
\end{equation*}
\end{corollary}
\begin{proof}[\hypertarget{loc:generalized_subdifferential_chain_rule.proof}Proof of \Cref{loc:generalized_subdifferential_chain_rule.statement}]

Denote $T_a = T_a(\dom(g),\range(A))$.
Note that $\forall x \in \mathbb{R}^n$,
\begin{equation*}
(g\circ A)(x)= ([g + \indicator_{T_a}] \circ A)(x),
\end{equation*}
because $T_a \supseteq \dom(g) \cap \range(A)$. So it suffices to calculate the subdifferential of the r.h.s. of the above equation.

If $\dom(g)  \cap \range(A) = \emptyset,$
then the result holds trivially because $g(Ax) \equiv \infty$, and $(g + \indicator_{T_a})(Ax) \equiv \infty$ as well.
Otherwise, by \Cref{loc:rint_qualification_holds_in_joint_supporting_subspace.statement}, 
$\rint \dom(g + \indicator_{T_a})  \cap \range(A)  \neq \emptyset$,
which lets us apply \cite[Theorem 23.9]{rockafellarConvexAnalysis1970}
to the composition $[g + \indicator_{T_a}] \circ A$.
Additionally, note that the only face of $\range(A)$ is itself, whence by
\Cref{loc:the_joint_supporting_subspace_reveals_faces.statement}
$\range(A) \subseteq T_a$. Since $0 \in \range(A)$,
this in fact implies that $T_a = T$, and the result follows.
\end{proof}
\subsection{Infimal convolution}
\label{loc:body.main_results.infimal_convolution}
Recall that for two convex functions $f,g:\mathbb{R}^n \to \bar{\mathbb{R}}$, the infimal convolution is
\begin{equation*}
(f\,  \square g)(z) := \inf_{x, y : x + y = z} f(x) + g(y).
\end{equation*}
In the following results, we denote by $f^*$ the convex conjugate of $f$.
\begin{corollary}[Attained infimal convolution]
\label{loc:attained_infimal_convolution_in_the_conjugate_of_a_sum.statement}
For two proper convex functions $f, g:\mathbb{R}^n \to \bar{\mathbb{R}}$,
and $T_a := T_a(\dom(f), \dom(g))$,
\begin{equation*}
(f+g)^* = (f + \indicator_{T_a})^* \square (g + \indicator_{T_a})^*,
\end{equation*}
where the infimum in the infimal convolution is attained.
\end{corollary}
\begin{proof}[\hypertarget{loc:attained_infimal_convolution_in_the_conjugate_of_a_sum.proof}Proof of \Cref{loc:attained_infimal_convolution_in_the_conjugate_of_a_sum.statement}]

Whenever $\dom(f) \cap \dom(g) =  \emptyset$, we have that $T_a = \emptyset$, in which case the equality
holds trivially, with the infimum attained at any $y \in \mathbb{R}^n$. 
Otherwise, by \Cref{loc:rint_qualification_holds_in_joint_supporting_subspace.statement}
the result follows from \cite[Theorem 20.1]{rockafellarConvexAnalysis1970}.
\end{proof}
\subsection{Convex optimization}
\label{loc:body.main_results.convex_optimization}
We next examine the following general convex program.
\begin{problem}[Convex program]
\label{loc:optimality_condition_for_the_exact_fr_dual.program}
Given lower semicontinuous proper convex functions $f:\mathbb{R}^n\to \bar{\mathbb{R}}$, $g:\mathbb{R}^m \to \bar{\mathbb{R}}$, and a matrix $A  \in \mathbb{R}^{m \times n}$,
\begin{equation*}
\minimize_{x \in \mathbb{R}^n} f(x)+ g(Ax).
\end{equation*}
\end{problem}
A standard regularity condition for this problem is that
\begin{equation*}
A(\rint \dom(f)) \cap \rint \dom(g) \neq \emptyset,
\end{equation*}
which is a weakening of Slater's condition.
A program with this property is called \emph{strongly consistent}~\cite[Theorem 31.2]{rockafellarConvexAnalysis1970}.

The next result shows the sense in which \Cref{loc:optimality_condition_for_the_exact_fr_dual.program}
can be localized to the joint facial subspace $T_a$, guaranteeing strong consistency whenever the program is feasible,
while retaining all the crucial information of the original program.
\begin{theorem}[Strong consistency of the localized program]
\label{loc:strong_consistency_of_the_localized_program.statement}
For $f:\mathbb{R}^n\to \bar{\mathbb{R}}$ and $g:\mathbb{R}^m \to \bar{\mathbb{R}}$ lower
semicontinuous proper convex functions,
a matrix $A  \in \mathbb{R}^{m  \times n}$, $T_a := T_a(A \dom(f),  \dom(g))$,
and any $x \in \mathbb{R}^n$,
\begin{equation}
\label{eq:myfn:equiv}
f(x)+ g(Ax) = (f(x) + \indicator_{A^{-1}(T_a)}(x)) + (g(x) + \indicator_{T_a}(x)).
\end{equation}
And whenever $\dom(f + g \circ A) \neq \emptyset$, 
\begin{equation*}
A\rint\dom(f + \indicator_{A^{-1}(T_a)}) \cap \rint\dom(g + \indicator_{T_a}) \neq \emptyset.
\end{equation*}
\end{theorem}
\begin{proof}[\hypertarget{loc:strong_consistency_of_the_localized_program.proof}Proof of \Cref{loc:strong_consistency_of_the_localized_program.statement}]

By \Cref{loc:containment_of_the_intersection_in_the_joint_supporting_subspace.statement},
$A \dom(f) \cap \dom(g) \subseteq T_a$,
and therefore $\indicator_{T_a}(Ax) = 0$ for all $x \in \dom(f) \cap A^{-1} \dom(g)$.
Additionally,
\begin{equation*}
\dom(f) \cap A^{-1} \dom(g) \subseteq A^{-1}(A\dom(f) \cap \dom(g)) \subseteq A^{-1} T_a,
\end{equation*}
which means that $\indicator_{A^{-1} T_a}(x)$ is $0$ on $\dom(f + g \circ A)$.
And for any $x \notin \dom(f+ g \circ A)$, it is easily checked that
both sides of \Cref{eq:myfn:equiv} evaluate to $\infty$, so \Cref{eq:myfn:equiv} holds for
all $x \in \mathbb{R}^n$.

For the second part of the result,
\begin{align}
&A\rint\dom(f + \indicator_{A^{-1}(T_a)}) \cap \rint(\dom(g + \indicator_{T_a})) \label{eq:ali:1}\\
 = &\rint(A\dom(f + \indicator_{A^{-1}(T_a)})) \cap \rint(\dom(g + \indicator_{T_a})) \label{eq:ali:2}\\
 = &\rint(A(\dom(f) \cap A^{-1}(T_a))) \cap \rint(\dom(g) \cap T_a) \label{eq:ali:3}\\
 = &\rint(A\dom(f) \cap T_a) \cap \rint(\dom(g) \cap T_a) \neq \emptyset. \label{eq:ali:4}
\end{align}
where \Cref{eq:ali:2} holds by \cite[Theorem 6.6]{rockafellarConvexAnalysis1970};
\Cref{eq:ali:4} uses the identity that $A(S_1 \cap  A^{-1}(S_2)) =  A S_1 \cap S_2$
for any sets $S_1 \subseteq \mathbb{R}^n$, $S_2 \subseteq \mathbb{R}^m$;
and that the set in \Cref{eq:ali:4} is non-empty is guaranteed by 
\Cref{loc:rint_qualification_holds_in_joint_supporting_subspace.statement}.
\end{proof}
We present an exact Fenchel--Rockafellar dual for \Cref{loc:optimality_condition_for_the_exact_fr_dual.program}:
a dual problem for which strong duality always holds and the dual attained.
\begin{corollary}[Exact Fenchel-Rockafellar dual]
\label{loc:exact_fenchel:rockafellar_dual.statement}
For $f:\mathbb{R}^n\to \bar{\mathbb{R}}$ and $g:\mathbb{R}^m \to \bar{\mathbb{R}}$ l.s.c. proper convex functions, and a matrix $A  \in \mathbb{R}^{m \times n}$,
let $T_a := T_a(A \dom(f), \dom(g))$. Then
\begin{equation*}
\inf_{x \in \mathbb{R}^n} [f(x)+ g(Ax)]=-\min_{y \in \mathbb{R}^m} [(f + \indicator_{A^{-1}(T_a)})^* (A^T y) + (g + \indicator_{T_a})^*(-y)].
\end{equation*}
\end{corollary}
\begin{proof}[\hypertarget{loc:exact_fenchel:rockafellar_dual.proof}Proof of \Cref{loc:exact_fenchel:rockafellar_dual.statement}]

The result follows from
\Cref{loc:strong_consistency_of_the_localized_program.statement}
with the use of \cite[Corollary 31.2.1]{rockafellarConvexAnalysis1970}.
\end{proof}
\begin{definition}[Minimizer]
\label{loc:minimizer.statement}
A \emph{minimizer} of a given convex program is any point $x \in \mathbb{R}^n$ that attains the infimum of the program, provided that the infimum is finite. 
\end{definition}
Next, we provide a characterization of the minimizers of 
\Cref{loc:optimality_condition_for_the_exact_fr_dual.program}, i.e.,
KKT optimality conditions. Our KKT condition 
holds without any
further assumptions on the program; we do not require strong consistency.
\begin{corollary}[KKT conditions]
\label{loc:optimality_condition_for_the_exact_fr_dual.statement}
Consider \Cref{loc:optimality_condition_for_the_exact_fr_dual.program}, and let $T_a := T_a(A \dom(f),  \dom(g))$.
A point $x \in \mathbb{R}^n$ is a minimizer of \Cref{loc:optimality_condition_for_the_exact_fr_dual.program} if, and only if,
there exists some $y \in \mathbb{R}^m$ such that
\begin{enumerate}
\item $Ax \in \partial(g + \indicator_{T_a})^*(y).$
\item $A^Ty \in \partial(f + \indicator_{A^{-1}(T_a)})(x)$.
\end{enumerate}
\end{corollary}
\begin{proof}[\hypertarget{loc:optimality_condition_for_the_exact_fr_dual.proof}Proof of \Cref{loc:optimality_condition_for_the_exact_fr_dual.statement}]

From 
\Cref{loc:strong_consistency_of_the_localized_program.statement},
the result follows from application of \cite[Corollary 31.3.1]{rockafellarConvexAnalysis1970}
to the functions $f + \indicator_{A^{-1}(T_a)}$ and $g +
\indicator_{T_a}$.
\end{proof}
\section{A framework for local face lattices}
\label{loc:body.a_framework_for_local_face_lattices}
Inspired by the lexicographic characterization of convex faces
of~\cite{martinez-legaz_LEXICOGRAPHICALCHARACTERIZATIONFACES_}, we
develop a framework describing the lattice of convex
faces at a point $x \in C$, which we will then use to show
the results of \Cref{loc:body.the_joint_supporting_subspace}.
We begin with a characterization of
faces of convex sets. In the next few results, we denote $X$ for
a real vector space.
\begin{proposition}[Characterization of convex faces by supporting subspaces]
\label{loc:characterization_of_convex_faces_by_supporting_subspaces.statement}
Let $C \subseteq X$ be a convex set.
A subset $F \subseteq C$ is a face of $C$ if, and only if
the two following conditions hold.
\begin{enumerate}
\item  $\aff(F) \cap C=F$,
\item $C \setminus F$ is convex.
\end{enumerate}
\end{proposition}
\begin{proof}[\hypertarget{loc:characterization_of_convex_faces_by_supporting_subspaces.proof}Proof of \Cref{loc:characterization_of_convex_faces_by_supporting_subspaces.statement}]

Let $F$ be a face. We argue by contradiction that $\aff(F)  \cap [C \setminus F] = \emptyset$;
let us assume that there is some
$z \in \aff(F)  \cap [C\setminus F]$.

If $F$ is a singleton, then $\aff(F) = F$, and $\aff(F) \cap [C \setminus F]= \emptyset$. 
Otherwise, we may assume that $0 \in F \setminus \{z\}$ without
loss of generality, because the faces of a convex set are invariant under translation. With this assumption,
$\aff(F)= \Span(F) = \cone(F-F)$. Then $z \in \cone(F-F)$, 
so there are $y_1, y_2 \in F$ such that $z = \alpha(y_1- y_2)$ for some $\alpha \geq 0$.
We can rule out the possibility that $z \in [y_1, y_2]$,
because by convexity of $F$ it would follow that $z \in F$.
The remaining possibility is that $y_1 \in [y_2, z[$.
Further, note that $y_1 \neq y_2$, because otherwise $z = 0$,
which disagrees with our assumption that $z \in F\setminus \{0\}$.
Therefore, $y_1 \in ]y_2, z[$. The line segment $[y_2, z]$ has a point $y_1 \in ]y_2, z[ \cap F$. Since $y_2, z \in C$,
the fact that $F$ is a face implies that
$z$ must then be in $F$, in contradiction
with our definition of $z$. We have therefore shown the claim.

We next argue that $C \setminus F$ is convex. $\forall y, z \in C\setminus F$, any $x \in ]y, z[$ is in $C$ by convexity of $C$. Also, $x$ is not in $F$ because $F$ is a face. So any $x \in ]y, z[$ is also in
$C\setminus F$, which means that $C\setminus F$ is convex.

We now show the converse: let $F \subseteq C$ be such that $C \setminus F$ is convex and $\aff(F) \cap C = F$, we show that $F$ is a face of $C$.

Take any $y, z \in C$, $x \in F \cap ]y, z[$. If $y \in \aff(F)$ then $z \in \aff(F)$ as well, because $z = \alpha(x-y) + y$ for some $\alpha>0$. So either $z, y$ are both in $\aff(F)$, or neither are. But
if neither are, then $z, y \in C\setminus \aff(F)$, and by convexity of $C \setminus \aff(F)$,
$x \in C\setminus \aff(F)$ also, which is a contradiction with $x \in F$. Therefore, $z, y$ are both in $F$, 
which shows that $F$ is a face.
\end{proof}
\Cref{loc:characterization_of_convex_faces_by_supporting_subspaces.statement}
motivates the next mathematical object.
\begin{definition}[Supporting subspace]
\label{loc:supporting_subspace.statement}
Let $C \subseteq X$ be convex. A supporting subspace $L \subseteq X$ of $C$ is an affine subspace such that $C \setminus L$ is convex.
\end{definition}
Though supporting subspaces are affine subspaces, we omit the
adjective ``affine" for brevity.
This notion appears in \cite[Exercise 5.4]{brondstedIntroductionConvexPolytopes1983},
and is related to order ideals, a concept defined for convex cones. 
Supporting subspaces are of central importance to this paper; it allows us to shift our focus away from the faces themselves to the
comparatively simpler affine subspaces which select them through
intersections with the convex set. The next result describes the relationship between supporting subspaces and faces.
\begin{proposition}[Relation of supporting subspaces to faces]
\label{loc:relation_of_supporting_subspaces_to_faces.statement}
Let $C \subseteq X$ be convex.
\begin{enumerate}
\item For any affine subspace $L \subseteq X$,
\begin{equation*}
L \text{ is a supporting subspace }  \iff L \cap C \text{ is a face.}
\end{equation*}
\item For any subset $F \subseteq C$,
\begin{equation*}
F \text{ is a face of }C  \iff \exists L \text{ a supporting subspace s.t.} F = L \cap C.
\end{equation*}
\end{enumerate}
\end{proposition}
\begin{proof}[\hypertarget{loc:relation_of_supporting_subspaces_to_faces.proof}Proof of \Cref{loc:relation_of_supporting_subspaces_to_faces.statement}]
\noindent\textbf{1) ($\implies$)}
We show that if $L$ is a supporting subspace, then $F := C \cap L$ satisfies both conditions of \Cref{loc:characterization_of_convex_faces_by_supporting_subspaces.statement}:
\begin{enumerate}
\item That $\aff(F) \cap C = F$: note that $\aff(F) \subseteq L$ because $F \subseteq L$, and so
$\aff(F) \cap C \subseteq L \cap C =: F$.
On the other hand, $L \cap C =  (L \cap C) \cap C \subseteq \aff(L \cap C) \cap C$, and so 
$F \subseteq \aff(F) \cap C$.
\item That $C \setminus F$ is convex: note that 
\begin{equation*}
C \setminus F = C \setminus [C  \cap L] = C \setminus L,
\end{equation*}
which is convex by definition of $L$ as a supporting subspace.
\end{enumerate}

\noindent\textbf{1) ($\impliedby$)}
Since $L \cap C$ is a face, $C \setminus (L \cap C)$ is convex by
\Cref{loc:characterization_of_convex_faces_by_supporting_subspaces.statement}, and therefore
$C\setminus L$ is convex, since it is the same set.

\noindent\textbf{2) ($\implies$)}
If $F$ is a face, then we argue that $\aff(F)$ is a suitable 
supporting subspace. Indeed, by \Cref{loc:characterization_of_convex_faces_by_supporting_subspaces.statement} we know that $F = C  \cap \aff(F)$ and that $C \setminus F$ is convex. Note that $C \setminus F = C \setminus (C \cap \aff(F))= C \setminus \aff(F)$, and therefore $C \setminus \aff(F)$ is convex, while $F = C\setminus \aff(F)$, 
which shows that $\aff(F)$ is a supporting subspace that selects
$F$ by intersection.

\noindent\textbf{2) ($\impliedby$)}
We are given a supporting subspace $L$ such that $F = L \cap C$. Since $\aff(F) \subseteq L$,
$\aff(F) \cap C \subseteq L \cap C$.
On the other hand, $L \cap C = F \subseteq \aff(F) \cap C$ because $F \subseteq C$.
Therefore, $F = \aff(F) \cap C$, the first of the two properties in \Cref{loc:characterization_of_convex_faces_by_supporting_subspaces.statement}.
For the second property, $C \setminus F = C \setminus (C \cap L) = C \setminus L,$
and therefore $C \setminus F$ is convex.
\end{proof}
In this work, we focus on ``local" supporting subspaces; the
subset of supporting subspaces that contain a point $x \in C$.
These subsets have some additional nice algebraic
properties. We find it useful to generalize this idea of ``locality" to
supporting subspaces that are localized at a set $S \subseteq C$ instead
of only at points, i.e., we consider all supporting subspaces of $C$ that also contain $S$.
To emphasize this connection, we say that a supporting subspace is ``at $S$" to describe its containment
of $S$. The reader should keep in
mind the important special case of $S$ being a singleton $S = \{x\}$,
for which the results are localized at a point.

Any supporting subspace $L$ at $S$ has its affine offset determined by
$S$. Recall that $L-L$ is the tangent space of $L$, i.e., the translation
of $L$ that becomes a (non-affine) subspace. With this, it becomes
clear that $L = (L-L) + S$. We can always deduce the correct affine
offset from $S$ alone. Therefore, we find it convenient to describe the
set of supporting subspaces that contain $S$ through their
tangent spaces. We henceforth call the tangent spaces of supporting subspaces 
``supporting subspaces" as well, and whether we mean 
the affine subspace or its tangent space will be clear from context.
\begin{definition}[Supporting subspace at S]
\label{loc:supporting_subspace_at_s.statement}
Let $C \subseteq X$ be a convex set and $S \subseteq X$ a non-empty subset. 
We call a subspace $U \subseteq X$ a \emph{supporting subspace} of $C$ at $S$ if:
\begin{enumerate}
\item $S - S \subseteq U$ 
\item $U + S$ is an (affine) supporting subspace of $C$ as defined in \Cref{loc:supporting_subspace.statement}.
\end{enumerate}

We define $\mathcal{F}(C; S)$ to be the set of all subspaces $U$
satisfying both conditions above for fixed $C$ and $S$.
\end{definition}
When $S$ is the singleton $\{x\}$, the requirement that $S-S \subseteq U$ is
trivially satisfied. When $S$ is not a singleton, this assumption guarantees
that $\forall s_1, s_2 \in S, U + s_1 =  U + s_2$. This implies that
$\forall s \in S, U+s = U + S$, and so the Minkowski sum $U + S$ is
a simple affine offset of $U$, for which $S \subseteq U+S$. In other
words, the supporting subspace $U + S$ contains $S$.
We introduce the notion of a \emph{nested normal cone}. Let $C \subseteq \mathbb{R}^n$ be convex and $U \subseteq \mathbb{R}^n$ be a subspace, then $\forall x \in \mathbb{R}^n$,
\begin{equation*}
N_C^U(x) := N_{C \cap (x + U)}(x) \cap U.
\end{equation*}
We would also like for our normal cones to located \emph{at a set} $S \subseteq C$, similarly to supporting subspaces.
\begin{equation*}
N_C(S) := \{v \in \mathbb{R}^n \mid \forall s \in S, c \in C,  \langle c-s, v\rangle \le 0\}.
\end{equation*}
The hyperplanes corresponding to the vectors in this normal cone are
supporting hyperplanes of $C$ that contain $S$. Note also that $N_C(S) = \bigcap_{s \in S} N_C(s).$
The idea of normal cones that are located at sets appears in the earlier work  
\cite{lu_NormalFansPolyhedral_2008}, where normal cones are defined for faces of
polytopes, instead of at points. 
Combining the two generalizations, we get the following.
\begin{definition}[Nested normals]
\label{loc:nested_normals.statement}
Let $C \subseteq \mathbb{R}^n$ be a convex set, and $S$ a
non-empty subset of $C$. Then for a subspace $U \subseteq \mathbb{R}^n$
such that $S-S \subseteq U$, and with $C' := C \cap (U + S)$,
\begin{equation*}
N_C^U(S) := \{v \in U \mid \forall c \in C', s \in S, \langle c-s, v\rangle \le 0\}
\end{equation*}
is a \emph{nested normal cone} of $C$ at $S$,
and we label any $v \in N_C^U(S)$
 a \emph{nested normal} of $C$ at $S$ in $U$.
\end{definition}
The normal cone at a set has an elegant
connection to the notion of relative interior of that set: the normal
cone at points in the relative interior match the normal cone of the set.
\begin{lemma}[Normal cones in the relative interior are representative]
\label{loc:normal_cones_in_the_relative_interior_are_representative.statement}
Let $C \subseteq \mathbb{R}^n$ be convex and $S \subseteq C$ a non-empty convex set.
Let $U \subseteq \mathbb{R}^n$ be a subspace such that $S-S \subseteq U$,
and $x \in \rint(S)$. Then
\begin{equation*}
N_C^U(S) = N_C^U(x).
\end{equation*}
\end{lemma}
\begin{proof}[\hypertarget{loc:normal_cones_in_the_relative_interior_are_representative.proof}Proof of \Cref{loc:normal_cones_in_the_relative_interior_are_representative.statement}]

That $N_C^U(S) \subseteq N_C^U(x)$ follows from simple inclusion $x \in S$.

For the containment ($\supseteq$), WLOG let $U = \mathbb{R}^n$.
We can make this assumption because the normal
cones are invariant under the translation $U + S \to U$,
and so we have an equivalent normal cone inside the subspace $U$. Since $U$
is a finite dimensional vector space, it is isomorphic to $\mathbb{R}^n$
for some $n \in \mathbb{N}$. 

Let $v \in N_C(x)$. Then $\forall s \in S, \langle v, s\rangle \le \langle v, x\rangle$ because $S \subseteq C$. 
We show that this inequality is in fact always an equality.
Arguing by contradiction, assume that there is an $s \in S$ such that $\langle v, s\rangle < \langle v, x\rangle$. Then since $x \in \rint(S)$, there is a $\varepsilon>0$
such that $(\varepsilon B + x) \cap \aff(S) \subseteq S$ for $B$ the unit
ball in $\mathbb{R}^n$. There is
a $\lambda > 0$ such that $s' := x + \lambda(s-x) \in \varepsilon B + x$.
Then note that $s' \in (\varepsilon B + x) \cap \aff(S) \subseteq S$, and
\begin{equation*}
\langle v, s'\rangle =  \langle v, x\rangle - \lambda \langle v, s-x\rangle.
\end{equation*}
The second term is strictly positive, therefore $\langle v, s'\rangle> \langle v, x\rangle$.
But this is a contradiction with $v \in N_C(x)$.
Therefore, we find that $\forall s \in S,$ $\langle v, s\rangle = \langle v, x\rangle$.
Then $\forall c \in C, s \in S$ $\langle c-s, v\rangle =  \langle c-x, v\rangle \le 0$,
and therefore $v \in N_C(S)$.
\end{proof}
Next, we show that the nested normals \emph{at a set} $S$ are orthogonal with $S - S$.
\begin{lemma}[Normal cone everywhere is the orthogonal subspace]
\label{loc:normal_cone_everywhere_is_the_orthogonal_subspace.statement}
Let $C \subseteq \mathbb{R}^n$ be a convex set, and $U \subseteq \mathbb{R}^n$ a subspace. Then for any non-empty $S \subseteq C$ such that $S-S \subseteq U$,
\begin{equation*}
N_C^U(S) \subseteq (S - S)^\perp \cap U.
\end{equation*}
Further, when $S = C \cap (U+S)$, the sets are equal.
\end{lemma}
\begin{proof}[\hypertarget{loc:normal_cone_everywhere_is_the_orthogonal_subspace.proof}Proof of \Cref{loc:normal_cone_everywhere_is_the_orthogonal_subspace.statement}]

Let $y \in N_C^U(S)$. By \Cref{loc:nested_normals.statement}, $\forall z, x \in S,$ since $x \in C$ and $z \in S$, $\langle x-z, y\rangle \le 0,$
and since $x \in S$ and $z \in C$, $\langle x-z, y\rangle \ge 0$. Then $y \perp S-S$, which shows the inclusion.

Next we show the containment ($\supseteq$), in the case that $S = C \cap (U+S)$.
Let $y \in (S-S)^\perp \cap U$. For any $x \in S, z \in C \cap (U+S)$ we need to show that $\langle z-x, y\rangle \leq 0$.
Note that $z \in S$, so $x-z  \in S-S$, and by definition of $y$, $\langle z-x,y\rangle=0$, which proves the statement.
\end{proof}
The core property of nested normals is that they describe
the lattice structure of supporting subspaces, which is closely related to the 
lattice structure of the faces. We make this idea precise in the next two results. 
\begin{theorem}[Nested normals induce supporting subspaces]
\label{loc:nested_normals_induce_supporting_subspaces.statement}
Let $C \subseteq \mathbb{R}^n$ be convex, $S \subseteq C$ non-empty, $U \subseteq \mathbb{R}^n$ a subspace, and $v \in N_C^U(S)$. Let $V \in \mathcal{F}(C;S)$ be such that $V \subseteq U$. Then 
\begin{equation*}
V \cap v^\perp \in \mathcal{F}(C;S).
\end{equation*}
\end{theorem}
\begin{proof}[\hypertarget{loc:nested_normals_induce_supporting_subspaces.proof}Proof of \Cref{loc:nested_normals_induce_supporting_subspaces.statement}]

First, note that if $v=0$, then $V \cap v^\perp=V$, and there result holds. We now consider the case where $v \neq 0$.

By \Cref{loc:normal_cone_everywhere_is_the_orthogonal_subspace.statement}, $v \perp S-S$, and therefore 
$S-S \subseteq V \cap v^\perp$. 

We make two observations. 
First, since $v \in N_C^U(S)$ and $V \subseteq U$, $\forall y \in C \cap (V+S)$, $\forall s \in S,$ we have that $\langle y, v\rangle \le \langle s, v\rangle$. Second,
this inequality is tight precisely for those values of $y$ that are also contained in $v^\perp + S$, i.e.,
\begin{align*}
&\{y \in C \cap (V+S) \mid \forall s \in S, \langle y, v\rangle = \langle s, v\rangle\}\\
=&((V +S) \cap C) \cap (v^\perp + S) = (V \cap v^\perp + S) \cap C.
\end{align*}
These two facts imply that 
\begin{equation*}
(V \cap v^\perp + S) \cap C = \argmax_{y \in C \cap (V+S)} \langle y, v\rangle.
\end{equation*}
The r.h.s. is an exposed face of the face $C \cap (V+S)$, so it is itself a
face, which by
\Cref{loc:relation_of_supporting_subspaces_to_faces.statement}
implies that $V \cap v^\perp  \in \mathcal{F}(C;S)$.
\end{proof}
We next show that all supporting subspaces are ``reached" by nested normals
in supporting subspaces.
\begin{theorem}[Completeness of the composition of nested normals]
\label{loc:completeness_of_the_composition_of_nested_normals.statement}
Let $C \subseteq \mathbb{R}^n$ be convex, $S \subseteq C$ non-empty,
and $V, U \in \mathcal{F}(C;S)$ such that $U \subseteq V$.
Then there is a sequence $V_0, \ldots, V_k \in \mathcal{F}(C;S)$ 
with $k= \dim(V)-\dim(U)$ such that $V_0=V$, $V_k=U$, and for each
$i \in \{ 1, \ldots,k\}$ there is a $v_i \in N_C^{V_{i-1}}(S) \setminus \{0\}$ such that $V_i = V_{i-1} \cap v_i^\perp$.
\end{theorem}
This result is similar to the lexicographic characterization of the faces
of convex sets in \cite[Proposition 6]{martinez-legaz_LEXICOGRAPHICALCHARACTERIZATIONFACES_},
which states that every face can be obtained by a sequence of faces
where each face is an exposed face of the preceding face. Let us state
a technical lemma that we will need for the proof.
\begin{lemma}[Affine slices preserve supporting subspaces]
\label{loc:affine_slices_preserve_supporting_subspaces.statement}
Let $C \subseteq \mathbb{R}^n$ be convex and $S \subseteq C$ non-empty.
Let $U \in \mathcal{F}(C;S)$, and $V$ any subspace such that
$S-S \subseteq V$. Then $U \cap V \in \mathcal{F}(C \cap (V+S);S)$.
\end{lemma}
\begin{proof}[\hypertarget{loc:affine_slices_preserve_supporting_subspaces.proof}Proof of \Cref{loc:affine_slices_preserve_supporting_subspaces.statement}]

Since $C \setminus (U + S)$ is convex, 
$(V+S)  \cap (C \setminus (U + S))$ is also convex. 
We can write this set as $[(V+S) \cap C] \setminus [(V+S) \cap (U+S)]$,
which is the same as $((V+S) \cap C) \setminus ((V \cap U)+S)$, because $S-S \in V \cap U$.
Since this set is convex, and $S-S \subseteq V  \cap U$, we find that $V \cap U \in \mathcal{F}((V +S) \cap C;S)$.
\end{proof}
\begin{proof}[\hypertarget{loc:completeness_of_the_composition_of_nested_normals.proof}Proof of \Cref{loc:completeness_of_the_composition_of_nested_normals.statement}]

It suffices to show the following statement: given any $U,V \in \mathcal{F}(C;S)$ such that $V \supsetneq U$,
there exists a $v \in N_C^V(S) \setminus \{0\}$ such that $v \perp U$. 
If this statement is shown to hold, applying it recursively generates the sequence of supporting subspaces $V_1, \ldots, V_k$ in a finite number of steps. 
Note that the dimension of the supporting subspaces decreases by one at each iteration while always containing $U$, which implies that $U \subseteq V_k$ with $\dim(U)=\dim(V_k)$, and therefore $U=V_k.$

We now show the above statement.
Let $C' = C \cap (V + S)$, $\bar{C} = \proj_{U^\perp} C'$, and $\bar{U} =  \proj_{U^\perp} (U + S)$.
Note that $\bar{C}$ is a convex set and $\bar{U}$
a singleton $\{x\} \subseteq \bar{C}$.  We argue that $\bar{U}$ is a face of $\bar{C}$.
Arguing by contradiction, we hypothesize the $\bar{U}$ is not a face of $\bar{C}$, in which case there are $\bar{y}, \bar{z} \in \bar{C}\setminus \bar{U}$
such that $x = \lambda\bar{y} + (1-\lambda)\bar{z}$ for $\lambda \in (0, 1)$. This implies that there are elements $y, z \in C'$ such that $\bar{y}= \proj_{U^\perp}y$ 
and $\bar{z}= \proj_{U^\perp}z$, with $w =  \lambda y+(1-\lambda)z$. Note that the line segment $[y,z]$ has endpoints $y, z \in C' \setminus (U+S)$ with a point $w  \in ]y,z[$  that is also contained in $C' \cap (U+S)$. This implies that $C' \cap (U+S)$ is not a face of $C'$, which, by \Cref{loc:relation_of_supporting_subspaces_to_faces.statement}, means that
$U \notin \mathcal{F}(C';S)$. This is a contradiction because in fact, $U  \in \mathcal{F}(C';S)$, as we now show. We know that $U  \in \mathcal{F}(C;S)$, so by
\Cref{loc:affine_slices_preserve_supporting_subspaces.statement}, $U = U \cap V \in \mathcal{F}(C \cap (V+ S);S)$, i.e., $U \in \mathcal{F}(C';S)$.
Therefore, we indeed have a contradiction, from which we find that $\bar{U}$ is a face of $\bar{C}$.

Note that $\proj_{U^\perp}(V + S) \subseteq V + S$. This becomes
apparent from the following equality:
$V + S = (V + S) + U = \proj_{U^\perp} (V + S) + U,$
where the first equality holds because $U \subseteq V$ implies that $V = V + U$. 

Then since $U+S, C \cap (V +S ) \subseteq V + S$, we find that $\bar{U}, \bar{C} \subseteq (V+S) \cap U^\perp$.
Since $\bar{U}$ is a face of $\bar{C}$, 
there exists a supporting hyperplane $H$ of $\bar{C}$, where
the hyperplane is relative to the affine subspace $(V+S) \cap U^\perp$, such that
$\bar{U} \subseteq H$.
Indeed, if there is no such hyperplane $H$, $\bar{U}$ would then be a singleton
in the interior of $\bar{C}$ (where the interior is relative to $(V+S)\cap U^\perp$, i.e., with the
induced topology in $(V+S)\cap U^\perp$).

Take $\bar{v} \in V \cap U^\perp$ to be the unit
normal to $H$ that is also in $N_{\bar{C}}^{V \cap U^\perp}(x)$. 
Then note that $\bar{v} \in N_C^{V}(S) \cap U^\perp$, which is the vector
we had to find.
\end{proof}
Together, \Cref{loc:nested_normals_induce_supporting_subspaces.statement}
and \Cref{loc:completeness_of_the_composition_of_nested_normals.statement} 
enables a powerful tool that is central to our proofs: we can
perform structural induction on the tree of supporting subspaces, using only nested normals.
Indeed, the set of supporting subspaces forms a tree, where the nodes
are the supporting subspaces, the root is
$\mathbb{R}^n$ (note that it is always a supporting subspace, and that
it contains any other supporting subspace), and every edge of the tree is
associated with a nested normal. Crucially, \Cref{loc:completeness_of_the_composition_of_nested_normals.statement}
shows that starting from the root, and generating the tree by using the nested
normals, we attain all supporting subspaces. Induction on this tree
then takes the following form.
\begin{corollary}[Induction by nested normal]
\label{loc:induction_by_nested_normal.statement}
Let $C$ be a convex set, $S \subseteq C$ a non-empty subset, and a statement $A:\mathcal{F}(C;S) \to \{\text{true}, \text{false}\}$. If the following statements hold: 
\begin{enumerate}
\item $A(\mathbb{R}^n),$
\item  $\forall U \in \mathcal{F}(C; S)$, $v \in N^U_C(S)$, $A(U) \implies A(U \cap v^\perp)$,
then it follows that
\begin{equation*}
\forall V \in \mathcal{F}(C;S), \, A(V).
\end{equation*}
\end{enumerate}
\end{corollary}
\begin{proof}[\hypertarget{loc:induction_by_nested_normal.proof}Proof of \Cref{loc:induction_by_nested_normal.statement}]

Consider any $U \in \mathcal{F}(C;S)$, and let the statement $A$ satisfy both
stated conditions. Then by \Cref{loc:completeness_of_the_composition_of_nested_normals.statement},
we have supporting subspaces $V_0, \ldots, V_k \in \mathcal{F}(C;S)$
and nested normals $v_1, \ldots,  v_k$ as in \Cref{loc:completeness_of_the_composition_of_nested_normals.statement}. $A(V_0)= \mathrm{true}$ 
because $V_0 = \mathbb{R}^n$. Also, $\forall i \in [k]$, $A(V_{i-1}) \implies
A(V_i)$ because of the existence
of the nested normal $v_i \in N_C^{V_{i-1}}(S) \setminus \{0\}$ with the property that 
$V_i= V_{i-1} \cap v_i^\perp$. Then $A(U)$ holds by classical induction,
because $U = V_k$.
\end{proof}
\subsection{Minimal supporting subspace}
\label{loc:narrative_for_conditional_subdifferentials.statement.minimal_supporting_subspace}
Recall that the smallest face of a convex set $C$ containing some $x \in C$
is $F_C(x)$, the face of $C$ \emph{generated} by $x$. It is the only face
with $x$ in its relative interior. In this section, we develop
an analogous notion for supporting subspaces, and characterize this
supporting subspace by using nested normals.
\begin{definition}[The generated supporting subspace]
\label{loc:generated_supporting_subspace.statement}
Given a convex set $C \subseteq \mathbb{R}^n$ with $S \subseteq C$ non-empty, 
we define the \emph{generated supporting subspace}
$H_C(S)$ of $C$, generated by $S$, to be the supporting subspace
that is contained in every other supporting subspace in $\mathcal{F}(C;S)$.
\end{definition}
We show existence and uniqueness for the above definition, starting with a lemma.
\begin{lemma}[Supporting subspaces are closed under intersections]
\label{loc:supporting_subspaces_are_closed_under_intersections.statement}
For $C \subseteq \mathbb{R}^n$ convex and $S \subseteq C$ non-empty, let $U, V \in \mathcal{F}(C; S)$. Then $U  \cap V \in \mathcal{F}(C;S)$.
\end{lemma}
\begin{proof}[\hypertarget{loc:supporting_subspaces_are_closed_under_intersections.proof}Proof of \Cref{loc:supporting_subspaces_are_closed_under_intersections.statement}]

The subspace $U \cap V$ is such that
\begin{equation*}
(U \cap V + S) \cap C = ((U + S) \cap C) \cap ((V + S) \cap C).
\end{equation*}
The r.h.s. is the intersection of two convex faces of $C$. Their intersection
is then also a face (immediate from the definition of faces).
By \Cref{loc:characterization_of_convex_faces_by_supporting_subspaces.statement}, $C \setminus [(U \cap V + S) \cap C]$
is convex. The convexity of this set, which we can write as $C \setminus ((U \cap V) + S)$,
implies that $U \cap V \in \mathcal{F}(C;S)$.
\end{proof}
\begin{proposition}
\label{loc:existence_and_uniqueness_of_the_generated_supporting_subspace.statement}
The generated supporting subspace of \Cref{loc:generated_supporting_subspace.statement} exists and is unique. 
\end{proposition}
\begin{proof}[\hypertarget{loc:existence_and_uniqueness_of_the_generated_supporting_subspace.proof}Proof of \Cref{loc:existence_and_uniqueness_of_the_generated_supporting_subspace.statement}]

For existence, starting at any supporting subspace $U \in \mathcal{F}(C;S)$,
either $U$ contains no other supporting subspace in
$\mathcal{F}(C;S)$, or we have a $U' \in \mathcal{F}(C;S)$ such that
$U' \subsetneq U$. Repeating this argument, we eventually find
a supporting subspace that contains no other, because the dimension is reduced
at every step, and the ambient dimension is finite.

For uniqueness, let $V_1,  V_2 \in \mathcal{F}(C;S)$ be such that 
they properly contain no other 
supporting subspace in $\mathcal{F}(C;S)$. Then by
\Cref{loc:supporting_subspaces_are_closed_under_intersections.statement}, $V_1 \cap V_2 \in \mathcal{F}(C;S)$ is properly contained in both $V_1$ and $V_2$, which is a contradiction.
\end{proof}
As might be expected, the generated supporting subspace is closely related to the generated face.
\begin{lemma}[Relation between the generated face and the generated supporting subspace]
\label{loc:relation_between_the_generated_face_and_the_generated_supporting_subspace.statement}
With $C \subseteq \mathbb{R}^n$ convex and $S \subseteq C$ non-empty,
\begin{enumerate}
\item $F_C(S) = C  \cap (H_C(S) + S).$
\item $H_C(S) = \Span(F_C(S)-F_C(S)).$
\end{enumerate}
\end{lemma}
\begin{proof}[\hypertarget{loc:relation_between_the_generated_face_and_the_generated_supporting_subspace.proof}Proof of \Cref{loc:relation_between_the_generated_face_and_the_generated_supporting_subspace.statement}]

By \Cref{loc:characterization_of_convex_faces_by_supporting_subspaces.statement},
$H_C(S)$ being the smallest supporting subspace implies that $C \cap (H_C(S) + S)$
is the smallest face containing $S$, because any properly contained, smaller face
would have its own supporting subspace (by \Cref{loc:relation_of_supporting_subspaces_to_faces.statement}), which would have to be properly contained in $H_C(S)$.

For the second statement, $\Span(F_C(S)-F_C(S))$ is a supporting subspace because $\Span(F_C(S)-F_C(S)) + S = \aff(F_C(S))$,
which is a supporting
subspace, as implied by \Cref{loc:characterization_of_convex_faces_by_supporting_subspaces.statement}.
Further, $\Span(F_C(S)-F_C(S))$ is contained in any other supporting subspace in
$\mathcal{F}(C;S)$, because given $U \in \mathcal{F}(C;S)$ such that
$U \subsetneq \Span(F_C(S)-F_C(S))$, we would then have $C  \cap (U + S)$ as a face of
$C$ containing $S$ that is properly contained in $F_C(S)$, in contradiction with
the definition of $F_C(S)$.
\end{proof}
We can now restate \Cref{loc:the_joint_supporting_subspace.statement} with
specialized terms.
\begin{definition}[Reformulated definition of the joint supporting subspace]
\label{loc:the_joint_supporting_subspace.technical}
The \emph{joint supporting subspace} of two convex sets $C, D \subseteq \mathbb{R}^n$, as defined in \Cref{loc:the_joint_supporting_subspace.statement}, is
\begin{equation*}
T(C, D) := H_{C-D}(0).
\end{equation*}
\end{definition}
The next result characterizes generated supporting subspaces in terms of nested normals.
\begin{theorem}[The generated supporting subspace is the only supporting subspace with no nested normals]
\label{loc:the_generated_supporting_subspace_is_the_only_supporting_subspace_with_no_nested_normals.statement}
Given a convex set $C \subseteq \mathbb{R}^n$ with $S \subseteq C$ non-empty, the \emph{generated supporting subspace} $H_C(S)$
is the only element in $\mathcal{F}(C;S)$ 
such that $N_C^{H_C(S)}(S) = \{0\}$.
\end{theorem}
\begin{proof}[\hypertarget{loc:the_generated_supporting_subspace_is_the_only_supporting_subspace_with_no_nested_normals.proof}Proof of \Cref{loc:the_generated_supporting_subspace_is_the_only_supporting_subspace_with_no_nested_normals.statement}]

The generated supporting subspace does not contain any non-trivial nested normals, because any non-trivial nested normal $v \in N_C^{H(C)}(S) \setminus \{0\}$ induces the supporting subspace $H_C(S) \cap v^\perp \subsetneq H_C(S)$ by \Cref{loc:nested_normals_induce_supporting_subspaces.statement}.

For any supporting subspace $V \in \mathcal{F}(C;S)$ such that $N_C^V(S)= \{0\}$,
we argue that $V = H_C(S)$. Note that $H_C(S) \subseteq V$ by \Cref{loc:generated_supporting_subspace.statement}.
The inclusion cannot be proper, because otherwise, by \Cref{loc:completeness_of_the_composition_of_nested_normals.statement}, there would be a nontrivial nested normal $v \in N_C^V(S)\setminus \{0\}$.
\end{proof}
\subsection{Nested normals of a difference of sets}
\label{loc:narrative_for_conditional_subdifferentials.statement.nested_normals_of_a_difference_of_sets}
We characterize the supporting subspaces of a difference of sets $C-D$ from the
supporting subspaces and nested normals of the individual sets $C$ and $D$ at $C
\cap D$. One might hope for the simple statement that $\mathcal{F}(C;S) \cap
\mathcal{F}(D;S) = \mathcal{F}(C-D;0)$ for $S \subseteq C \cap D$. Unfortunately, though $(\supseteq)$
turns out to hold, $(\subseteq)$ does not in general. 
\begin{example}[Mismatched nested normals]
\label{loc:example_mismatched_nested_normals.statement}
Let $h:\mathbb{R} \to  \mathbb{R}, h(x):= x^2$, $C := \epi\min(\indicator_{\mathbb{R}_{-}},  h)$
and $D :=  \mathrm{hypo}\max(-\indicator_{\mathbb{R}_{-}},  -h)$, where $\mathrm{hypo}$ denotes the hypograph of the concave function. Then $\{0\}$ is a
face of both $C$ and $D$, yet $C-D = \mathbb{R} \times \mathbb{R}_{+}$, and
therefore $\{0\} \notin \mathcal{F}(C-D;0)$. 
\end{example}
There is therefore a need for some
additional condition on supporting subspaces
shared by $C$ and $D$ for them to also be
supporting subspaces of $C-D$. The next result
allows us to derive such a condition by using nested normals.
\begin{theorem}[Decomposition of supporting subspaces of differences]
\label{loc:decomposition_of_supporting_subspaces_of_differences.statement}
Let $C, D \subseteq \mathbb{R}^n$ be convex with $S \subseteq C \cap D$ non-empty, and $U \in \mathcal{F}(C-D;0)$. Then
\begin{enumerate}
\item $U \in \mathcal{F}(C;S) \cap \mathcal{F}(D;S).$
\item $N^U_{C-D}(0) = N^U_C(S) \cap (-N^U_D(S)).$
\item $C \cap (U+S) - D \cap (U+S) = (C-D) \cap U.$
\end{enumerate}
\end{theorem}
The next corollary shows that the missing condition is the existence of a synchronized ``path" of nested normals for $C$ and $D$.
\begin{corollary}[Characterization of the faces of differences of convex sets]
\label{loc:characterization_of_the_faces_of_differences_of_convex_sets.statement}
Let $C, D \subseteq \mathbb{R}^n$ be convex and $S \subseteq C \cap D$ non-empty.
The elements of $\mathcal{F}(C-D;0)$ are precisely the elements 
$U \in \mathcal{F}(C;S) \cap \mathcal{F}(D;S)$
for which there exists a sequence $V_0, \ldots, V_k \in \mathcal{F}(C;S) \cap  \mathcal{F}(D;S)$
such that $V_0=\mathbb{R}^n$, $V_k = U$, and
such that for each $i \in \{ 1, \ldots,k\}$ there is a $v_i \in N_C^{V_{i-1}}(S) \cap -N_D^{V_{i-1}}(S) \setminus \{0\}$ such that $V_i = V_{i-1} \cap v_i^\perp$.
\end{corollary}
\begin{proof}[\hypertarget{loc:characterization_of_the_faces_of_differences_of_convex_sets.proof}Proof of \Cref{loc:characterization_of_the_faces_of_differences_of_convex_sets.statement}]

That any $U$ with such a sequence is in $\mathcal{F}(C-D;0)$ follows by first noticing that each $v_i$ in the sequence belongs to
$N^U_{C-D}(0)$ by \Cref{loc:decomposition_of_supporting_subspaces_of_differences.statement}. By applying \Cref{loc:nested_normals_induce_supporting_subspaces.statement} sequentially, it follows that $U  \in \mathcal{F}(C-D;0)$.

That any element $U \in \mathcal{F}(C-D;0)$ has such a sequence follows from \Cref{loc:completeness_of_the_composition_of_nested_normals.statement}, which shows that there is a sequence of nested normals of $C-D$ at $0$ that, starting from $\mathbb{R}^n$, attains $U$. By \Cref{loc:decomposition_of_supporting_subspaces_of_differences.statement}, this sequence corresponds to a sequence $v_i,  \ldots , v_k$ as described in the statement.
\end{proof}
We can now better understand \Cref{loc:example_mismatched_nested_normals.statement}.
The supporting subspace $\{0\}$ of $C$ and $D$ is obtained via the sequence of nested normals
$\{(0,  -1), (1, 0)\}$ for $C$ and $\{(0, 1), (1, 0)\}$ for $D$. Only the
first nested normals are opposite to each other, and therefore there is no
synchronized path of nested normals to $\{0\}$. If we modify the example by
taking $D :=
\mathrm{hypo} \max(-\indicator_{\mathbb{R}_+},  - h)$,
then $\{0\}$ is attained by $\{(0, 1),  (-1, 0)\}$, and the sequence is
synchronized with that of $C$. In this modified example,
$\{0\}$ is indeed a supporting subspace of $C-D$ because $C-D$ is a convex set
such that
$C-D \subseteq \mathbb{R} \times  \mathbb{R}^+$ and $C-D \cap (\mathbb{R} \times \{0\}) = \mathbb{R}_{-} \times \{0\}$, which shows that $\{0\} \in \mathcal{F}(C-D;0)$.

To show \Cref{loc:decomposition_of_supporting_subspaces_of_differences.statement}, we require the following lemma.
\begin{lemma}[Normal cone of the difference of convex sets]
\label{loc:normal_cone_of_the_difference_of_convex_sets.statement}
Given convex sets $C, D \subseteq \mathbb{R}^n$ with $S \subseteq C \cap D$ non-empty,
\begin{equation*}
N_{C-D}(0) = N_C(S)  \cap (- N_D(S)).
\end{equation*}
\end{lemma}
\begin{proof}[\hypertarget{loc:normal_cone_of_the_difference_of_convex_sets.proof}Proof of \Cref{loc:normal_cone_of_the_difference_of_convex_sets.statement}]

We start with the inclusion $( \subseteq )$.
Any $v  \in N_{C-D}(0)$ has the property that $\forall c \in C, d \in D$, $\langle v, (c-d)-0\rangle \le 0$.
Then $\forall c  \in C, s \in S$, $\langle v, c-s\rangle \leq 0$ because $s \in D$, and therefore $v \in N_C(S)$. By a similar argument, $v \in -N_D(S)$, which shows the inclusion.

For the containment ($\supseteq$), fix some $s \in S$.
let $v \in N_{C}(S) \cap (-N_{D}(S))$. Then $\forall c \in C, d \in D$ 
\begin{equation*}
\langle v, c-d\rangle= \langle v, c-s\rangle  + \langle v, s-d\rangle \leq 0,
\end{equation*}
and so $v \in N_{C-D}(0)$.
\end{proof}
\begin{proof}[\hypertarget{loc:decomposition_of_supporting_subspaces_of_differences.proof}Proof of \Cref{loc:decomposition_of_supporting_subspaces_of_differences.statement}]

We argue by nested normal induction (from \Cref{loc:induction_by_nested_normal.statement}) over the supporting subspaces of $C-D$. Fixing some $S \subseteq C\cap D$, let $A(U)$ be the statement that all three points in the statements hold for supporting subspace $U \in \mathcal{F}(C-D;0)$.

The base case, where $U = \mathbb{R}^n$, holds because $\mathbb{R}^n$ is always a supporting subspace. Point 2) is given by \Cref{loc:normal_cone_of_the_difference_of_convex_sets.statement}, and point 3) holds trivially.

For the inductive step, let $V \in \mathcal{F}(C-D;0)$ be such that $A(V)$ is true. Let $v \in N_{C -D}^V(0)\setminus \{ 0 \}$, and $U = V  \cap v^\perp$. In the rest of the proof, we show that $A(U)$ is true. Note that by point 2) of $A(V)$, $v \in N_C^V(S) \cap (-N_D^V(S))$, which by \Cref{loc:nested_normals_induce_supporting_subspaces.statement} implies that $U \in \mathcal{F}(C;S) \cap \mathcal{F}(D;S)$, and this shows point 1) of $A(U)$.

We next show that for $C'=C\cap (V+S)$, $D' = D \cap (V+S)$,
\begin{equation}
\label{eq:identity_u}
(C' - D') \cap U = (C' \cap (U+S))-(D' \cap (U+S)).
\end{equation}
We start with the inclusion ($\subseteq$): Take $z \in C'$, $y \in D'$ with $z-y \in U$.
Then since $U \subseteq v^\perp$, $\langle z-y, v\rangle = 0$, and therefore $\forall s \in S$, $\langle z-s, v\rangle - \langle y-s, v\rangle = 0$.
That $\langle z-s, v\rangle \leq 0$ follows from $v \in N_{C}^V(S)$ since $z  \in C' \subseteq C$ and $z-s \in V$, and similarly, $\langle y-s, v\rangle \ge 0$ because $v \in N_D^V(S)$.
But then the only way that $\langle z-s, v\rangle - \langle y-s, v\rangle = 0$
is if $\langle z-s, v\rangle = 0$ and $\langle y-s, v\rangle = 0$ individually.
It follows that $z, y \in U + S$, and so the inclusion holds.
The containment ($\supseteq$) is trivial, so we have shown \Cref{eq:identity_u}.

By point 3) of $A(V)$, $C'-D' = (C -D) \cap V$. Also, since $U \subseteq V$, $C' \cap (U+S)$ = $C \cap (U+S)$, and similarly for $D$. Applying these identities on both sides of \Cref{eq:identity_u}, we find that
\begin{equation*}
(C - D) \cap U = C \cap (U+S) - D \cap (U+S).
\end{equation*}
This shows that point 3) of $A(U)$ is true.

To show point 2) of $A(U)$, we compute
\begin{align}
N_C^U(S) \cap (-N_D^U(S)) &= [N_{C \cap (U+S)}(S) \cap (-N_{D \cap (U+S)}(S))] \cap U \label{eq:normals:1}\\
&= N_{C \cap (U+S) - D \cap (U+S)}(0) \cap U \label{eq:normals:2}\\
&=N_{(C-D)\cap U}(0) \cap U = N_{C-D}^U(0), \label{eq:normals:3}
\end{align}
where \Cref{eq:normals:1} holds by definition of nested normals, \Cref{eq:normals:2} follows from \Cref{loc:normal_cone_of_the_difference_of_convex_sets.statement} applied to the sets $C  \cap (U+S)$ and $D \cap (U+S)$, and \Cref{eq:normals:3} follows from point 3) of $A(U)$. This shows that point 2) of $A(U)$ holds for $U$, and the proof is complete.
\end{proof}
We now have the tools to prove \Cref{loc:containment_of_the_intersection_in_the_joint_supporting_subspace.statement}.
\begin{proof}[\hypertarget{loc:containment_of_the_intersection_in_the_joint_supporting_subspace.proof}Proof of \Cref{loc:containment_of_the_intersection_in_the_joint_supporting_subspace.statement}]

By \Cref{loc:the_joint_supporting_subspace.technical},
$T(C, D) = H_{C-D}(0)$. So $T \in \mathcal{F}_{C-D}(0)$, which
by \Cref{loc:decomposition_of_supporting_subspaces_of_differences.statement} implies that,
taking $S = C \cap D$, $T \in \mathcal{F}_C(C \cap D) \cap \mathcal{F}_D(C \cap D)$.
By the definition of supporting subspaces (\Cref{loc:supporting_subspace_at_s.statement}),
$C \cap D - C \cap D \subseteq T$.
\end{proof}
\section{Proofs}
\label{loc:body.proofs}
We begin with the proof of \Cref{loc:recovery_of_the_classical_results.statement}, after which each subsection is dedicated to
one result from \Cref{loc:body.the_joint_supporting_subspace}.
\begin{proof}[\hypertarget{loc:recovery_of_the_classical_results.proof}Proof of \Cref{loc:recovery_of_the_classical_results.statement}]

First we show that for $C \subseteq \mathbb{R}^n$ convex, and $x \in \rint(C)$, $F_C(x)= C$. 
Note that
\begin{equation*}
N_C^{\Span(C-C)}(x) = N_C^{\Span(C-C)}(C) = (C-C)^\perp \cap \Span(C-C)= \{0\},
\end{equation*}
where the first equality follows by \Cref{loc:normal_cones_in_the_relative_interior_are_representative.statement} and the second by \Cref{loc:normal_cone_everywhere_is_the_orthogonal_subspace.statement}.
Then by \Cref{loc:the_generated_supporting_subspace_is_the_only_supporting_subspace_with_no_nested_normals.statement}, $H_C(x) = \Span(C-C)$, and by \Cref{loc:relation_between_the_generated_face_and_the_generated_supporting_subspace.statement}, $F_C(x) = (\Span(C-C) + x) \cap C = C$.

From \Cref{loc:the_joint_supporting_subspace_reveals_faces.statement}, for $x \in \rint(\dom(f)) \cap \rint(\dom(g))$ and $T_a = T_a(\dom(f), \dom(g))$,
\begin{equation*}
T_a \cap \dom(f) = F_{\dom(f)} (\dom(f) \cap \dom(g)) \supseteq F_{\dom(f)}(x) = \dom(f).
\end{equation*}
From this the results follows.
\end{proof}
\subsection{Regularity in the joint supporting subspace }
\label{loc:body.proofs.regularity_in_the_joint_supporting_subspace}
We now show the main property of the joint supporting subspace, 
\Cref{loc:qualification_conditions_always_hold_in_the_joint_supporting_subspace.statement}.
We first state and prove a lemma which connects the notion of separation
of convex set to a statement reminiscent of the qualification condition
by Mordukhovich in variational analysis: \cite[Theorem
2.19]{mordukhovich_VariationalAnalysisApplications_2018}.
\begin{lemma}[Separation from local normal cones]
\label{loc:separation_from_local_normal_cones.statement}
For convex sets $C, D \subseteq \mathbb{R}^n$ and $S \subseteq C \cap D$ nonempty, the following equivalence holds:
\begin{equation}
\label{eq:main_sep}
N_C(S) \cap (-N_D(S)) \neq \{0\} \iff C \text{ is separated from } D.
\end{equation}
\end{lemma}
\begin{proof}[\hypertarget{loc:separation_from_local_normal_cones.proof}Proof of \Cref{loc:separation_from_local_normal_cones.statement}]

We start with $(\implies)$. There exists a $v \in [N_C(S) \cap (-N_D(S))] \setminus \{0\}$,
and from \Cref{loc:normal_cone_everywhere_is_the_orthogonal_subspace.statement} $v \perp S-S$, whence 
the hyperplane $v^\perp + S$ separates $C$ from $D$.

For the reverse direction $(\impliedby)$, suppose $C$ and $D$ are
separated by a hyperplane $H$. Then $S \subseteq H$, and there exist
$v \in \mathbb{R}^n \setminus \{0\}$ normal to $H$ and $\beta \in \mathbb{R}$
such that
\begin{enumerate}
\item $\langle v, s \rangle = \beta$ for all $s \in S$;
\item $\langle v, c \rangle \ge \beta$ for all $c \in C$;
\item $\langle v, d \rangle \le \beta$ for all $d \in D$.
\end{enumerate}

Thus, for all $y \in C$ and $s \in S$, $\langle v, y - s \rangle \ge 0$,
so $v \in N_C(S)$. Similarly, $-v \in N_D(S)$, whence
$v \in N_C(S) \cap (-N_D(S)) \neq \{0\}$.
\end{proof}
\begin{proof}[\hypertarget{loc:qualification_conditions_always_hold_in_the_joint_supporting_subspace.proof}Proof of \Cref{loc:qualification_conditions_always_hold_in_the_joint_supporting_subspace.statement}]

Recall that from \Cref{loc:the_joint_supporting_subspace.technical}, $T = H_{C-D}(0)$.
Then by
\Cref{loc:the_generated_supporting_subspace_is_the_only_supporting_subspace_with_no_nested_normals.statement},
$N_{C-D}^T(C \cap D) = \{0\}$. Applying
\Cref{loc:decomposition_of_supporting_subspaces_of_differences.statement}, we find that
$N^T_C(C \cap D) \cap (-N^T_D(C \cap D)) = \{0\}$.

By applying \Cref{loc:separation_from_local_normal_cones.statement}
in the space $T_a$, the result follows. Note that \Cref{loc:separation_from_local_normal_cones.statement} is stated in $\mathbb{R}^n$, yet we can apply it
in $T_a$ because $T$ is a finite-dimensional subspace, thereby isomorphic to $\mathbb{R}^n$, and
\Cref{loc:separation_from_local_normal_cones.statement}
is invariant when translating $T$ to $T_a$. 
\end{proof}
\subsection{Analytical characterizations}
\label{loc:body.proofs.analytical_characterizations}
In this section we prove \Cref{loc:the_joint_supporting_subspace_reveals_faces.statement} and 
\Cref{loc:characterization_through_generated_faces.statement}.
\begin{lemma}[Nested normals of constrained set]
\label{loc:nested_normals_of_constrained_set.statement}
Let $C \subseteq \mathbb{R}^n$ be a convex set and $S \subseteq C$ a
nonempty set. Let $U, V \subseteq \mathbb{R}^n$ be subspaces such that $S - S \subseteq U \subseteq V$ and $C \cap (V + S) \subseteq U + S$. Then
\begin{equation*}
N_C^U(S) + (U^\perp \cap V) = N_{C \cap (U + S)}^V(S).
\end{equation*}
\end{lemma}
\begin{proof}[\hypertarget{loc:nested_normals_of_constrained_set.proof}Proof of \Cref{loc:nested_normals_of_constrained_set.statement}]

We begin with the containment ($\supseteq$). 
Let $x \in N_C^V(S).$ Then 
$\forall s \in S, y \in C  \cap (V+S)$, note that
$y-s \in U$, and so
\begin{equation*}
\langle \proj_U x, y-s\rangle = \langle x, y-s\rangle \le 0.
\end{equation*}
This shows that $\proj_U x \in N_C^U(S)$. Now $x = \proj_U x + \proj_{U^\perp} x$,
and $x, \proj_{U} x \in V \implies \proj_{U^\perp}x \in V$, 
therefore $x \in N_C^U(S) + (U^\perp  \cap  V)$.

For the inclusion ($\subseteq$), we consider any $q \in N^U_C(S)$ and 
$z \in U^\perp \cap V$. Let $y \in C \cap (U+S), s \in S$, and note that
$\langle z,y-s\rangle=0$ because $y-s  \in U$. Then 
\begin{equation*}
\langle q+z,  y-s\rangle =  \langle q, y-s\rangle \le 0.
\end{equation*}
Therefore $q + z  \in N_{C \cap (U + S)}^V(S)$, which shows the inclusion and completes the proof.
\end{proof}
\begin{lemma}[Conical projections yield supporting subspaces]
\label{loc:conical_projections_yield_supporting_subspaces.statement}
Let $C \subseteq \mathbb{R}^n$ be convex, $S \subseteq C$ non-empty, 
and $U \in \mathcal{F}(C;S)$. Then for any $K \subseteq N^U_C(S)$,
\begin{equation*}
U \cap K^\perp \in \mathcal{F}(C;S).
\end{equation*}
\end{lemma}
\begin{proof}[\hypertarget{loc:conical_projections_yield_supporting_subspaces.proof}Proof of \Cref{loc:conical_projections_yield_supporting_subspaces.statement}]

Since $K$ is finite dimensional, there is a finite number of vectors $v_1, \ldots, v_m \in K$
such that $\Span(K) = \Span(v_1, \ldots, v_m)$.
Then since $v_1, \ldots , v_m \in N_C^U(S)$, by repeated application of
\Cref{loc:nested_normals_induce_supporting_subspaces.statement}, $U \cap v_1^\perp \ldots \cap v_m^\perp$ is a supporting subspace, which is the same set as $U \cap K^\perp$.
\end{proof}
\begin{lemma}[Generated supporting subspace of one of the two convex sets]
\label{loc:generated_supporting_subspace_of_one_of_the_two_convex_sets.statement}
For convex sets $C, D \subseteq \mathbb{R}^n$, $T=T(C,D)$, $T_a = T + C  \cap D$, we have
\begin{equation*}
H_C(C \cap D) = T \cap N_C^T(C \cap T_a)^\perp.
\end{equation*}
\end{lemma}
\begin{proof}[\hypertarget{loc:generated_supporting_subspace_of_one_of_the_two_convex_sets.proof}Proof of \Cref{loc:generated_supporting_subspace_of_one_of_the_two_convex_sets.statement}]

Let $U = T \cap N_C^T(C \cap T_a)^\perp$.
we first show that $U \in \mathcal{F}(C \cap D; C)$. Note that
$T \in \mathcal{F}(C \cap D; C)$ by \Cref{loc:decomposition_of_supporting_subspaces_of_differences.statement}.
Then we can use
\Cref{loc:conical_projections_yield_supporting_subspaces.statement}
because $N_C^T(C \cap T_a) \subseteq N_C^T(C \cap D)$, which shows that $U \in \mathcal{F}(C \cap D; C)$ as well. 

In the remainder of the proof we show that the assumption of
\Cref{loc:the_generated_supporting_subspace_is_the_only_supporting_subspace_with_no_nested_normals.statement}
holds for $U$, namely that $N_C^U(C \cap D) = \{0\}$. 

We first show that $N_C^U(C \cap T_a) = \{0\}$. By
\Cref{loc:normal_cone_everywhere_is_the_orthogonal_subspace.statement}, $N_C^T(C \cap T_a) = (C \cap T_a - C \cap T_a)^\perp \cap T$. 
Also, note that by the identity in 
\Cref{loc:nested_normals_of_constrained_set.statement},
we find that $N_C^U(C \cap T_a)= N_C^T(C \cap T_a) \cap U$.
Combining these two facts,
\begin{equation*}
N_C^U(C \cap T_a) = N_C^T(C \cap T_a) \cap U = (C \cap T_a - C \cap T_a)^\perp  \cap U.
\end{equation*}
But note that $C \cap T_a - C \cap T_a \subseteq U$, therefore it fo $N_C^U(C \cap T_a)= \{ 0 \}$.  That $C \cap T_a - C \cap T_a \subseteq U$ holds because $C \cap T_a - C \cap T_a \subseteq T$ and $(C \cap T_a - C \cap T_a) \perp N_C^T(C \cap T_a)$ by \Cref{loc:normal_cone_everywhere_is_the_orthogonal_subspace.statement}. 

It only remains to show that $N_C^U(C \cap D)=N_C^U(C \cap T_a)$.  By
\Cref{loc:rint_qualification_holds_in_joint_supporting_subspace.statement},
we may pick some $x \in \rint(C \cap T_a) \cap \rint(D \cap T_a)$. Then by
\Cref{loc:normal_cones_in_the_relative_interior_are_representative.statement},
\begin{equation*}
N_{C}^U(C \cap T_a) = N_{C}^U(x) \supseteq N_C^U(C \cap D).
\end{equation*}
Then $N_C^U(C \cap D)=\{ 0 \}$, and so the result follows from \Cref{loc:the_generated_supporting_subspace_is_the_only_supporting_subspace_with_no_nested_normals.statement}. 
\end{proof}
We are now in a position to prove the analytical characterizations
in \Cref{loc:body.the_joint_supporting_subspace}.
\begin{proof}[\hypertarget{loc:the_joint_supporting_subspace_reveals_faces.proof}Proof of \Cref{loc:the_joint_supporting_subspace_reveals_faces.statement}]

From
\Cref{loc:generated_supporting_subspace_of_one_of_the_two_convex_sets.statement},
$H_C(C \cap D) = T \cap N_C^T(C \cap T_a)^\perp$. Since $(C \cap T_a - C \cap T_a) \perp N_C^T(C \cap T_a)$ by \Cref{loc:normal_cone_everywhere_is_the_orthogonal_subspace.statement}, $(C \cap T_a - C \cap T_a) \subseteq H_C(C \cap D)$, i.e.,
$H_C(C \cap D) + C \cap T_a \supseteq C \cap T_a$. But $H_C(C \cap D) + C \cap T_a = H_C(C \cap D) + C \cap D$, therefore
\begin{equation}
\label{eq:cont}
H_C(C \cap D) + C \cap D \supseteq C \cap T_a.
\end{equation}

Additionally, since we know that $T \in \mathcal{F}(C \cap D;C)$ from \Cref{loc:decomposition_of_supporting_subspaces_of_differences.statement}, by definition of the generated supporting subspace we know that
$H_C(C \cap D) \subseteq T$. Therefore, 
\begin{equation}
\label{eq:inc}
H_C(C \cap D) + C \cap D \subseteq T_a.
\end{equation}

From \Cref{eq:cont} and \Cref{eq:inc} it follows that $(H_C(C \cap D) + C\cap D) \cap C = C \cap T_a$.
Then the statement follows from 
\Cref{loc:relation_between_the_generated_face_and_the_generated_supporting_subspace.statement}.
\end{proof}
\begin{proof}[\hypertarget{loc:characterization_through_generated_faces.proof}Proof of \Cref{loc:characterization_through_generated_faces.statement}]

From \Cref{loc:the_joint_supporting_subspace_reveals_faces.statement}, we have that
$T \in \mathcal{F}(C; C \cap D)$ is a supporting subspace of $C$ that contains $C  \cap D$,
and so by \Cref{loc:generated_supporting_subspace.statement},
$H_C(C \cap D) \subseteq T$. Therefore,
for $T':= H_C(C \cap D) + H_D(C \cap D),$ $T' \subseteq T.$

For the other containment, we argue by contradiction: let $T' \subsetneq T$. Then there is a unit vector $v \in T$
with $v \perp T'$. Then $v \perp H_C(C  \cap D)$, which by
\Cref{loc:generated_supporting_subspace_of_one_of_the_two_convex_sets.statement},
implies that $v \in N_C^T(C \cap T)$. By \Cref{loc:normal_cone_everywhere_is_the_orthogonal_subspace.statement},
 $-v \in N_C^T(C \cap T)$ as well.
And since $N_C^T(C \cap T) \subseteq N_C^T(C \cap D)$, we find that $v, -v \in N_C^T(C \cap D)$. 
By a similar argument, $v, -v \in N_D^T(C \cap D)$.
But then $v \in N_C^T(C \cap D)  \cap (-N_D^T(C \cap D))$, which by \Cref{loc:decomposition_of_supporting_subspaces_of_differences.statement} means that $v \in N_{C-D}^T(0)$.
But by \Cref{loc:the_generated_supporting_subspace_is_the_only_supporting_subspace_with_no_nested_normals.statement}, $T$ cannot contain such a nested normal because by \Cref{loc:the_joint_supporting_subspace.technical}, $T=H_{C-D}(0)$. Therefore, by contradiction $T=T'$.
\end{proof}
Next, we present the proofs for the corollaries of 
\Cref{loc:characterization_through_generated_faces.statement}, starting with a simple lemma.
\begin{lemma}[Formulation of the tangent space]
\label{loc:formulation_of_the_tangent_space.statement}
For a convex set $C \subseteq \mathbb{R}^n$ and $S \subseteq C$ non-empty,
\begin{equation*}
\Span(C-C) = \Span(C-S).
\end{equation*}
\end{lemma}
\begin{proof}[\hypertarget{loc:formulation_of_the_tangent_space.proof}Proof of \Cref{loc:formulation_of_the_tangent_space.statement}]

That $\Span(C-C) \supseteq \Span(C-S)$ is trivial, so
we only show $\Span(C-C) \subseteq \Span(C-S)$.
It suffices to show that $C-C \subseteq \Span(C-S)$;
let $u, v \in C$, then for any fixed $x \in S$,
\begin{equation*}
u-v = (u-x)-(v-x) \in \Span(C-S).
\end{equation*}
\end{proof}
\begin{proof}[\hypertarget{loc:joint_supporting_subspace_from_difference_of_faces.proof}Proof of \Cref{loc:joint_supporting_subspace_from_difference_of_faces.statement}]

In this proof, we denote $F_C := F_C(C \cap D)$ and $F_D := F_D(C \cap D)$.
\begin{align}
\Span(F_C-F_D) &= \Span(F_C-F_D) + (C \cap D - C \cap D) \label{eq:diff_face:1}\\
 &= \Span(F_C - F_D + (C \cap D  - C \cap D)) \label{eq:diff_face:2}\\
 &= \Span((F_C - C \cap D) -(F_D-C \cap D)) \label{eq:diff_face:3}\\
 &= \Span(F_C - C \cap D) +\Span(F_D-C \cap D) \label{eq:diff_face:4}\\
 &= \Span(F_C - F_C) + \Span(F_D- F_D), \label{eq:diff_face:5}\\
&= H_C(C \cap D) + H_D(C \cap D). \label{eq:diff_face:6}\\
\end{align}
where \Cref{eq:diff_face:1} holds because 
\begin{equation*}
(C \cap D - C \cap D) \subseteq F_C - F_D \subseteq  \Span(F_C-F_D),
\end{equation*}
and \Cref{eq:diff_face:5} follows from \Cref{loc:formulation_of_the_tangent_space.statement}.
The result then follows from 
\Cref{loc:characterization_through_generated_faces.statement}.
\end{proof}
\begin{proof}[\hypertarget{loc:joint_supporting_subspace_as_affine_hull.proof}Proof of \Cref{loc:joint_supporting_subspace_as_affine_hull.statement}]

Throughout this proof, we omit the argument $C \cap D$, i.e., 
$H_C := H_C(C \cap D)$, and similarly
for $H_D, F_C, F_D$. 
We also define $\bar{F}_C = F_C-C \cap D$ and $\bar{F}_D = F_D - C \cap D$.
By \Cref{loc:formulation_of_the_tangent_space.statement},
\begin{equation*}
H_C= \Span(F_C-F_C)= \Span(\bar{F}_C).
\end{equation*}
With this, starting with $T_a$ as in \Cref{loc:characterization_through_generated_faces.statement},
\begin{align*}
&T_a(C, D) = H_C + H_D + C \cap D\\  
&= \Span(\bar{F}_C) + \Span(\bar{F}_D) + C \cap D\\
&= \Span(\Span(\bar{F}_C) \cup \Span(\bar{F}_D)) + C \cap D\\
&= \Span(\bar{F}_C \cup \bar{F}_D) + C \cap D
= \aff(\bar{F}_C \cup \bar{F}_D) + C \cap D\\
 &= \aff(\bar{F}_C \cup \bar{F}_D + C \cap D) 
 = \aff(F_C \cup F_D).
\end{align*}
\end{proof}
\subsection{Second characterization}
\label{loc:body.proofs.second_characterization}
We can now state the second characterization in full.
\begin{theorem}[Iterative bilateral facial reduction]
\label{loc:iterative_bilateral_facial_reduction.statement}
Let $C, D \subseteq \mathbb{R}^n$ be convex sets and 
$S \subseteq C \cap D$ non-empty. Beginning with $T_0  = \mathbb{R}^n$, let
\begin{equation*}
T_{i+1} = T_i \cap (N^{T_i}_{C}(S) \cap [-N^{T_i}_{D}(S)])^\perp.
\end{equation*}
Let $\ell$ be the first index for which 
$T_{\ell+1} = T_\ell$. Then:
\begin{enumerate}
\item $T_\ell =  T(C, D)$.
\item $\ell \leq n$.
\end{enumerate}
\end{theorem}
\begin{proof}[\hypertarget{loc:iterative_bilateral_facial_reduction.proof}Proof of \Cref{loc:iterative_bilateral_facial_reduction.statement}]

First, $(N^{T_i}_{C}(S) \cap [-N^{T_i}_{D}(S)]) = N^{T_i}_{C-D}(S)$ by \Cref{loc:decomposition_of_supporting_subspaces_of_differences.statement}.
Applying \Cref{loc:conical_projections_yield_supporting_subspaces.statement} at each step, we find that $\forall i  \in \{ 0, \ldots, \ell \},\, T_i \in \mathcal{F}(C-D;S)$.
Finally, note that whenever $N^{T_i}_{C}(S) \cap [-N^{T_i}_{D}(S)] \neq \{0\}$,
$\dim(T_{i+1})\le \dim(T_i) - 1$. Therefore, the dimension may be reduced
at most $n$ times, which implies that $\ell \le n$. On the other hand, the fact that $T_\ell = T_{\ell+1}$
implies that $N^{T_\ell}_{C-D}(S) = \{0\}$.
Then from
\Cref{loc:the_generated_supporting_subspace_is_the_only_supporting_subspace_with_no_nested_normals.statement}
it follows that $T_\ell = H_{C-D}(S)$, which is $T(C,D)$ from \Cref{loc:the_joint_supporting_subspace.technical}.
\end{proof}
\begin{proof}[\hypertarget{loc:simplified_iterative_bilateral_facial_reduction.proof}Proof of \Cref{loc:simplified_iterative_bilateral_facial_reduction.statement}]

We consider the formula
\begin{equation*}
T_{i+1} = T_i \cap (N^{T_i}_{C}(x) \cap [-N^{T_i}_{D}(x)])^\perp.
\end{equation*}
from  \Cref{loc:iterative_bilateral_facial_reduction.statement}.
We show that the r.h.s. is the same set as
\begin{equation*}
T_i \cap (N_{C \cap (T_i + x)}(x) \cap [-N_{D \cap (T_i+x)}(x)])^\perp.
\end{equation*}
We use a substitution as in \Cref{loc:nested_normals_of_constrained_set.statement} with $V = \mathbb{R}^n$.
\begin{align*}
&T_i \cap (N_{C \cap (T_i + x)}(x) \cap [-N_{D \cap (T_i+x)}(x)])^\perp\\
= &T_i \cap ((N^{T_i}_{C}(x) + T_i^\perp) \cap [-N^{T_i}_{D}(x) + T_i^\perp])^\perp\\
= &T_i \cap (N^{T_i}_{C}(x) \cap [-N^{T_i}_{D}(x)] + T_i^\perp)^\perp\\
= &T_i \cap (N^{T_i}_{C}(x) \cap [-N^{T_i}_{D}(x)])^\perp.
\end{align*}
The last equality holds because $N^{T_i}_{C}(x) \cap [-N^{T_i}_{D}(x)] \subseteq T_i$, so we may use the
identity that $T_i \cap (S + T_i^\perp)^\perp = T_i \cap S^\perp$ 
for any set $S \subseteq T_i$.
\end{proof}
\subsection{Conclusion}
\label{loc:body.proofs.conclusion}
In this work we define and characterize the joint supporting subspace $T_a$, which exists between any two
convex sets. By defining this affine subspace from the domains of two convex functions, and then reducing each function to it, qualification conditions are guaranteed to hold. We obtain a number of characterizations for $T_a$ by bridging between the facial structure of the difference of two convex sets $C-D$ and the facial structure of the individual convex sets. By doing this, we show that ``pathology on the boundary" of convex functions can be tamed by considering the facial structure of their domains.

Future work avenues include specifying and implementing efficient reduction
algorithms for convex functions, and pairing them with optimization solvers to obtain robust algorithms with strong theoretical guarantees.
\subsection{Acknowledgements}
\label{loc:body.proofs.acknowledgements}
The author would like to thank Sharvaj Kubal, Naomi Graham, prof. Yaniv Plan, and prof.
Ozgur Yilmaz for many fruitful discussions. A special thanks to prof. Michael Friedlander for 
pointing out that a generalization of the Fenchel-Rockafellar dual and
optimality conditions would be
a promising research direction. This paper was produced with the Latex Exporter plugin in Obsidian.
\section*{Declarations}
The author is supported by the University of British Columbia.
\printbibliography
\end{document}

%% file: header.tex
\usepackage{amsmath}
\usepackage{amsthm}
\usepackage{amsfonts}
\usepackage{amssymb}
\usepackage{biblatex}
\usepackage{graphicx}
\usepackage{hyperref}
\usepackage{cleveref}
\usepackage{subcaption}
\usepackage{tikz}
\usepackage{authblk}

\theoremstyle{plain}
\newtheorem{theorem}{Theorem}[section]
\newtheorem{corollary}{Corollary}[section]
\newtheorem{lemma}{Lemma}[section]
\newtheorem{proposition}{Proposition}[section]

\theoremstyle{definition}
\newtheorem{definition}{Definition}[section]
\newtheorem{problem}{Problem}[section]
\newtheorem{example}{Example}[section]

\theoremstyle{remark}